\documentclass[11pt,a4paper,reqno]{amsart}
\usepackage[margin=3cm,includehead,includefoot]{geometry}
\usepackage{amsmath,amssymb,amsthm,mathrsfs,verbatim,enumitem,tikz}
\usetikzlibrary{shapes.misc,calc,intersections,patterns}
\newtheorem{theorem}{Theorem}
\newtheorem{lemma}[theorem]{Lemma}

\newtheorem{corollary}[theorem]{Corollary}
\newtheorem{conjecture}[theorem]{Conjecture}
\theoremstyle{definition}
\newtheorem{definition}[theorem]{Definition}
\newtheorem{remark}[theorem]{Remark}
\numberwithin{theorem}{section}

\setlist[itemize]{leftmargin=1cm}
\setlist[enumerate]{leftmargin=1.5cm}

\def\D{\mathcal{D}}

\def\H{\mathbb{H}}
\def\N{\mathbb{N}}
\def\P{\mathbb{P}}
\def\R{\mathbb{R}}
\def\T{\mathcal{T}}
\def\TT{\mathbb{T}}
\def\Z{\mathbb{Z}}
\def\U{\mathcal{U}}
\def\stab{\mathcal{S}}
\def\stabl{\mathcal{S}_L}
\def\stabu{\mathcal{S}_U}
\def\quasi{\mathcal{Q}}
\def\diam{\mathrm{diam}}
\def\<{\langle}
\def\>{\rangle}
\def\0{\mathbf{0}}
\renewcommand{\leq}{\leqslant}
\renewcommand{\geq}{\geqslant}

\title{Monotone cellular automata in a random environment}

\author{B\'ela Bollob\'as}
\address{Department of Pure Mathematics and Mathematical Statistics, Wilberforce Road, Cambridge, CB3 0WA, UK, and Department of Mathematical Sciences, University of Memphis, Memphis, TN 38152, USA, and London Institute for Mathematical Sciences, 35a South Street, London, W1K 2XF, UK}
\email{b.bollobas@dpmms.cam.ac.uk}

\author{Paul Smith}
\address{IMPA, 110 Estrada Dona Castorina, Jardim Bot\^anico, Rio de Janeiro, 22460-320, Brazil}
\email{psmith@impa.br}

\author{Andrew Uzzell}
\address{Department of Mathematics, Uppsala University, PO Box 480, SE-751 06 Uppsala, Sweden}
\email{andrew.uzzell@math.uu.se}

\subjclass[2010]{Primary 60K35; Secondary 60C05}
\date{\today}

\begin{document}

\begin{abstract}
In this paper we study in complete generality the family of two-state, deterministic, monotone, local, homogeneous cellular automata in $\Z^d$ with random initial configurations. Formally, we are given a set $\U=\{X_1,\dots,X_m\}$ of finite subsets of $\Z^d\setminus\{\0\}$, and an initial set $A_0\subset\Z^d$ of `infected' sites, which we take to be random according to the product measure with density $p$. At time $t\in\N$, the set of infected sites $A_t$ is the union of $A_{t-1}$ and the set of all $x\in\Z^d$ such that $x+X\in A_{t-1}$ for some $X\in\U$. Our model may alternatively be thought of as bootstrap percolation on $\Z^d$ with arbitrary update rules, and for this reason we call it \emph{$\U$-bootstrap percolation}.

In two dimensions, we give a classification of $\U$-bootstrap percolation models into three classes -- supercritical, critical and subcritical -- and we prove results about the phase transitions of all models belonging to the first two of these classes. More precisely, we show that the critical probability for percolation on $(\Z/n\Z)^2$ is $(\log n)^{-\Theta(1)}$ for all models in the critical class, and that it is $n^{-\Theta(1)}$ for all models in the supercritical class.

The results in this paper are the first of any kind on bootstrap percolation considered in this level of generality, and in particular they are the first that make no assumptions of symmetry. It is the hope of the authors that this work will initiate a new, unified theory of bootstrap percolation on $\Z^d$.
\end{abstract}

\maketitle

\section{Introduction}

Cellular automata are dynamical systems on graphs in which vertices have one of a finite number of states at discrete times $t$. The states of the vertices change according to an update rule that is homogeneous (the same rule applies to every vertex) and local (the rule depends on only a finite neighbourhood of the vertex under consideration). These systems were first introduced by von Neumann (see \cite{vN}, for example) at the suggestion of Ulam \cite{Ulam}. The subject of this paper is a substantial extension of, and has its roots in,  \emph{bootstrap percolation}, a particular type of cellular automaton proposed in $1979$ by Chalupa, Leath and Reich \cite{CLR}.

In the simplest and most widely-studied model of bootstrap percolation, called \emph{$r$-neighbour bootstrap percolation}, we are given a graph $G$ and a subset $A\subset V(G)$ of the vertices of $G$. The elements of $A$ are initially \emph{infected}, and at each time $t$, infected vertices stay infected, and among the uninfected vertices of $G$, those that have at least $r$ infected neighbours become infected. In keeping with the percolation literature, we use the words `vertex' and `site' interchangeably. The set of all vertices in $G$ that are eventually infected is called the \emph{closure} of $A$ and denoted $[A]$. If $A$ is such that $[A]=V(G)$ then we say that $A$ \emph{percolates} $G$, or simply that there is \emph{percolation}. 

The key problem is to analyse the behaviour of the $r$-neighbour model acting on an initial set $A$ chosen randomly according to a product measure with density $p$. In particular, one would like to know how large $p$ should be for there to be percolation with high probability (for a sequence of graphs with size of vertex set tending to infinity) or almost surely (for a single infinite graph). The results in relation to this phase transition are usually framed in terms of the \emph{critical probability}, defined as
\begin{equation}\label{eq:pcr}
p_c(G,r) := \inf \big\{ p \, : \, \P_p\big([A]=V(G)\big)\geq 1/2 \big\}.
\end{equation}
A considerable amount is now known about $p_c(G,r)$ in the case of \emph{lattice graphs}: a graph $G$ is a ($d$-dimensional) lattice graph if it is (isomorphic to) a (not necessarily planar) translation invariant locally finite graph with vertex set $\Z^d$. (Equivalently, $G$ is a lattice graph if (there is an isomorphism between $V(G)$ and $\Z^d$ under which) there exists a finite symmetric set $X\subset\Z^d\setminus\{\0\}$ such that the neighbourhood of any vertex $x$ is the set $x+X$.) The most natural lattice graph, and the one that has attracted the greatest amount of study, is the nearest neighbour graph on $\Z^d$. Indeed, the first rigorous result in the field of bootstrap percolation, which was proved by van Enter \cite{vanEnter}, was that $p_c(\Z^2,2)=0$. Schonmann \cite{Schon} later showed that $p_c(\Z^d,r)$ is equal to $0$ if $r\leq d$ and $1$ otherwise. Aizenman and Lebowitz~\cite{AL} were the first to recognize that there is much more to say about the model in the context of finite subgraphs of lattice graphs, and they proved that $p_c([n]^d,2)=\Theta(1)/(\log n)^{d-1}$. Holroyd \cite{Hol} went considerably further in the case $d=r=2$, proving that $p_c([n]^2,2)=(1+o(1))\pi^2/18\log n$. Vast generalizations of these results on $[n]^d$ for general $d$ and $r$ were obtained in the weak sense by Cerf and Cirillo \cite{CerfCir} and Cerf and Manzo \cite{CerfManzo}, and in the sharp sense by Balogh, Bollob\'as and Morris \cite{BBM3D} and Balogh, Bollob\'as, Duminil-Copin and Morris \cite{BBDCM}.

Bootstrap percolation has been studied on a number of other lattice graphs, not only on $\Z^d$ and $[n]^d$; for a small selection of these results, see \cite{BBhyp,BBMmaj,BPP,BisSchon,GG96,GG99,Schon2}. In particular, Gravner and Griffeath \cite{GG96,GG99} studied $r$-neighbour bootstrap percolation on general two-dimensional lattice graphs. (Gravner and Griffeath called their model `threshold dynamics', but it is easily seen to be equivalent to $r$-neighbour bootstrap percolation on two-dimensional lattices.) Unfortunately, not all of the results in \cite{GG96,GG99} were rigorously proved, and in some cases they were not even correct (see, for example, \cite{vEH}).

The first departure from the study of bootstrap percolation on lattice graphs came from Duarte \cite{Duarte}, who introduced a two-dimensional directed model (in fact, the introduction by Duarte of this model predates the introduction by Gravner and Griffeath of $r$-neighbour models on general lattices). He again took the nearest neighbour graph on $\Z^2$ with threshold $r=2$, but he took horizontal edges to be directed rightwards, and vertical edges to be directed in both directions (alternatively one could think of the vertical edges as being undirected, or replaced by two edges, one directed each way). Under this model, an uninfected vertex $x$ becomes infected if at least two of its in-neighbours are infected; that is, if we take $X$ to be the set $\{(-1,0),(0,1),(0,-1)\}$, then a site $x$ becomes infected if $x+X$ contains at least two infected vertices. It is easy to see that this model causes a drift in the growth of the infected set: new sites only become infected to the right of existing sites. The critical probability for the Duarte model was determined up to a constant factor by Mountford \cite{MountDuarte}; he showed that $p_c(\overrightarrow{\Z}_n^2,2)=\Theta\big((\log\log n)^2/\log n\big)$, where $\overrightarrow{\Z}^2_n$ is the nearest neighbour graph on the torus $\Z_n^2$ with horizontal edges directed rightwards and vertical edges undirected. The correct constant in place of the `$\Theta(\cdot)$', or even the existence of a constant, is still not known.

However, even general $r$-neighbour models on directed lattice graphs, without any additional assumptions, are far less general than the models we introduce and analyse in the present paper. Our aims are threefold:
\begin{itemize}
\item to define a new family of percolation models, which generalize all previously studied models of bootstrap percolation on $\Z^d$, and to derive a number of basic properties of these models in two dimensions;
\item to divide the two-dimensional models into three classes according to the nature of their large-scale behaviour, and to prove general results about the critical probabilities of models in two of these classes;
\item to pose a number of open questions about these new general models.
\end{itemize}
We hope through achieving these aims to initiate a new, unified understanding of all bootstrap percolation models on the square lattice.

We go considerably further than all previous authors in the following specific sense: rather than taking a single special model or a special class of models, we take all local, homogeneous, monotone models. This means that we fix an $s$ and take all neighbourhoods of $x$ that consist only of sites at distance at most $s$ from $x$; for each of these we say whether or not $x$ becomes infected given that all the sites in the neighbourhood are infected, subject to the condition that the neighbourhoods that infect $x$ form an up-set (this is part of what we mean by monotone; we also mean that infected vertices never become uninfected). The same rule is then applied to every site $x$. The formal definition is as follows.

\begin{definition}\label{de:update}
Fix an integer $d\geq 2$. Let $m\in\N$ and for each $i\in[m]$ let $X_i\subset\Z^d\setminus\{\0\}$ be finite and non-empty. Let $\U=\{X_1,\ldots,X_m\}$. We call $\U$ an \emph{update family} and each $X_i$ an \emph{update rule}. Let $A\subset\Z^d$. In \emph{$\U$-bootstrap percolation} with update family $\U$ and initial set $A=A_0$, for each $t\geq 0$ we set
\[
A_{t+1} := A_t \cup \{x\in\Z^2 : X_i+x\subset A_t \text{ for some } X_i\in\U\}.
\]
\end{definition}

Thus, there is a set $A$ of initially infected sites and at each time step, whenever a translation of a configuration $X_i$ by an element of $\Z^d$ is completely infected, the corresponding translation of the origin also becomes infected. The \emph{closure of $A$ under $\U$} is defined to be $[A]:=\cup_{t=0}^\infty A_t$. When the closure of $A$ is the whole of $\Z^d$, we say that $A$ \emph{percolates $\Z^d$ under $\U$}, or more usually, when $d$, $\U$ and $A$ are understood from the context, we say simply that there is \emph{percolation}.

The $r$-neighbour model is easily seen to be an example of $\U$-bootstrap percolation: one takes the update rules $X_1,\dots,X_m$ to be the $m:=\binom{|\mathcal{N}|}{r}$ subsets of size $r$ of the graph (in-)neighbourhood $\mathcal{N}$ of the origin. However, these are only a very special subclass of our models, and in general, our models look nothing like $r$-neighbour models. The family in Figure \ref{fi:family} is an example of the kind of general model we have in mind.

For the remainder of the paper, except for a discussion of open problems in Section~\ref{se:open}, we shall restrict our attention to the case $d=2$. Gravner and Griffeath \cite{GG96,GG99}, in their study of two-dimensional $r$-neighbour bootstrap percolation, took the underlying lattice graphs to be undirected. The result of this was that the bootstrap processes had a strong symmetry property, namely that $X\in\U$ if and only if $-X\in\U$. While this may seem natural, quite the opposite is true: symmetry vastly restricts the range and behaviour of the possible models. We have already seen a glimpse of this with the result of Mountford on the Duarte model, which shows that critical probabilities are not always negative powers of $\log n$. In Section~\ref{se:sketch} we explain in greater detail why symmetry is such a powerful property for a bootstrap process to have; here we mention briefly that the main reason is that it allows one to constrain growth to one-dimensional strips (and in higher dimensions, to lower-dimensional strips). Gravner and Griffeath imposed an additional restriction on their models, which had the effect of simplifying the analysis even further. The restriction was that the process should be `balanced', meaning in their case roughly that the four hardest directions in which to grow are equally hard. The reason why this constraint is useful is that, together with symmetry, it implies the existence of bounding parallelograms, all four of whose sides are equally difficult from which to grow, and which can be controlled one-dimensionally in two independent directions.

%The first, which is essentially just a constraint on $r$ for a given lattice graph, was that the process should be critical, meaning roughly, in their case, that growth is neither too easy (there exist finite sets with infinite closure) nor too hard (there exist closed cofinite sets). This is a natural assumption because non-critical models exhibit very different large-scale behaviour, and in the case of lattice graphs they are trivial to analyse. The second assumption, unlike the first, was not at all natural,

We emphasize that we impose \emph{none} of these additional constraints. In particular, most importantly, we do not impose any symmetry constraints on our models, and in general our models will not exhibit any symmetry at all. This is one of the key strengths of our results, and one of the key advances over results by other authors.

Let $\Lambda$ denote either $\Z^2$ or the discrete torus $\Z_n^2:=(\Z/n\Z)^2$. As most previous authors, we are interested in the random setting, in which sites of the lattice $\Lambda$ are included in the initial set $A$ independently with probability $p\in(0,1)$. We describe a set $A$ chosen in this way as \emph{$p$-random}, and we write $\P_p$ for the product probability measure. The $\U$-bootstrap percolation analogue of the critical probability for $r$-neighbour models defined in \eqref{eq:pcr} is as follows:
\begin{equation}\label{eq:pcU}
p_c(\Lambda,\U) := \inf \Big\{ p \, : \, \P_p\big(A \text{ percolates in $\U$-bootstrap percolation on } \Lambda\big) \geq 1/2 \Big\},
\end{equation}
where $A$ is a $p$-random subset of $\Lambda$.

\subsection{The classification of $\U$-bootstrap percolation processes in two dimensions}

For each $u\in S^1$, let $\H_u := \{ x\in\Z^2 \,:\, \< x,u \> < 0 \}$ be the discrete half-plane with boundary perpendicular to $u$ (we use $\<\cdot,\cdot\>$ to denote the usual Euclidean inner product). The following simple definition is one of the most important of the paper; it will form the basis of our tripartition of update families in Definition~\ref{de:class}.

\begin{definition}
A unit vector $u\in S^1$ is a \emph{stable direction for $\U$} if $[\H_u]=\H_u$. We denote by $\stab=\stab(\U)$ the set of all stable directions for $\U$ and we call $\stab$ the \emph{stable set for $\U$}. If $u\in\stab^c = S^1\setminus\stab$, then we call $u$ an \emph{unstable direction for $\U$}. A stable direction $u$ is said to be \emph{strongly stable for $\U$} if it is contained in an open interval of stable directions in $S^1$.
\end{definition}

%The definitions of stable directions and the stable set are not new: the significance of stable directions in relation to the large-scale behaviour of $r$-neighbour models on undirected lattice graphs was first recognized by Gravner and Griffeath \cite{GG93,GG96,GG99} and later utilized by Duminil-Copin and Holroyd \cite{DCH}. We emphasize, though, that this is the first time they have been studied in the context of completely general models.

%Although the stable set depends on the action of the $\U$-bootstrap process on an uncountable set of half-planes, it is actually rather straightforward to determine. In Section~\ref{se:stableset} we describe an easy polynomial time algorithm to compute the stable set of any given update family.

One of the key properties of the $\U$-bootstrap process that we establish in this paper is that the approximate large-scale behaviour of the process is governed primarily by the geometry of the stable set $\stab$ alone. This assertion is encapsulated by the following definition, which is our classification of $\U$-bootstrap percolation models. 

\begin{definition}\label{de:class}
Let $\U$ be a (two-dimensional) update family with stable set $\stab$. We say that $\U$ is
\begin{itemize}
\item \emph{supercritical} if there exists an open semicircle $C\subset S^1$ such that $\stab\cap C=\emptyset$;
\item \emph{critical} if $\stab\cap C\neq\emptyset$ for every open semicircle $C\subset S^1$ and there exists an open semicircle $C\subset S^1$ not containing any strongly stable directions;
\item \emph{subcritical} if every open semicircle $C\subset S^1$ contains a strongly stable direction.
\end{itemize}
\end{definition}

The most important of the three classes of $\U$-bootstrap models is the critical class, which includes all previously studied models of bootstrap percolation on $\Z^2$.

A similar classification in the special case of $r$-neighbour models is given by Gravner and Griffeath \cite{GG96,GG99}. They define a process to be supercritical if there exist finite subsets of $\Z^2$ with infinite closures, to be subcritical if there exist closed cofinite proper subsets of $\Z^2$, and to be critical otherwise. Our definitions coincide precisely with their definitions in the case of the small subclass of our models that they consider. Moreover, the distinction between supercritical and critical is the same in all cases, although not obviously so; in fact, this statement is the content of Theorem \ref{th:superclass}. The distinction between critical and subcritical in general is not the same, however: one can show that the Gravner-Griffeath definition is equivalent to the statement that every direction is stable, while we only require every open semicircle to contain a strongly stable direction. Theorem \ref{th:stabclass} implies that there is an infinite collection of families $\U$ that fall under our definition of subcritical but not under the Gravner-Griffeath definition.

Among the results of this paper, which we state formally next, we prove that the distinction made in Definition \ref{de:class} between supercritical and critical families is the right one, in the specific sense that the critical probabilities of supercritical families are polynomial in $t$ and the critical probabilities of critical families are polylogarithmic in $t$. We are not able to prove that the distinction between critical and subcritical families given here is also correct, although we conjecture that this is the case, in the specific sense that the critical probabilities of subcritical models should be bounded away from zero\footnote{Since the submission of this paper, Balister, Bollob\'as, Przykucki and Smith \cite{BBPS} have proved that this conjecture is true.} (see Conjecture~\ref{co:sub}). We defer further discussion until Section \ref{se:open}.

\subsection{Results}\label{se:results}

As stated earlier in the introduction, our aims in this paper are threefold: to define and prove basic results about $\U$-bootstrap percolation; to prove the first results about the critical probabilities of critical and supercritical models; and to pose a number of open questions about these new models.

Our main results concern the second of these aims. We state these results first in terms of the infection time of the origin (Theorems~\ref{th:main} and~\ref{th:super}), and then in their more usual (and essentially equivalent) form in terms of critical probabilities (Theorem~\ref{th:pc}). Given a set $A\subset\Z^2$, define the stopping time
\[
\tau = \tau(A,\U) := \min \{ t\geq 0 \,:\, \0\in A_t \}.
\]
Thus, $\tau$ is the time at which the origin becomes infected in $\U$-bootstrap percolation with initial set $A$. In the statements of all theorems (and throughout the paper), constants, implicit or otherwise, are quantities that depend only on $\U$.

The following is our main theorem.

\begin{theorem}\label{th:main}
%Let $\U$ be a critical two-dimensional bootstrap percolation update family. Then there exist $0<\alpha_1\leq\alpha_2$ such that
%Let $\U$ be a critical two-dimensional bootstrap percolation update family. Then there exist $0<\alpha_1\leq\alpha_2$ such that
Let $\U$ be a critical two-dimensional bootstrap percolation update family. Then, with high probability as $p\to 0$,
\[
%\frac{1}{p^{\alpha_1}} \leq \log\tau \leq \frac{1}{p^{\alpha_2}}
\tau = \exp\big(p^{-\Theta(1)}\big)
\]
%with high probability as $p\to 0$.
\end{theorem}

We also prove the following theorem for supercritical families.

\begin{theorem}\label{th:super}
Let $\U$ be a supercritical two-dimensional bootstrap percolation update family. Then, with high probability as $p\to 0$,
\[
\tau = p^{-\Theta(1)}.
\]
\end{theorem}

We remark that the lower bound of Theorem~\ref{th:super} is a triviality: there must be at least one initially infected site within distance $O(\tau)$ of the origin. The content of the theorem is the upper bound.
%The strength of these theorems lies in their level of generality and the lack of any assumptions about the update family beyond the necessary assumption that it is critical or supercritical.

As a corollary of Theorems~\ref{th:main} and~\ref{th:super}, we have that for any critical or supercritical update family and any $p>0$, a $p$-random subset of $\Z^2$ almost surely percolates.

\begin{corollary}\label{co:pcZ2}
Let $\U$ be a critical or supercritical two-dimensional bootstrap percolation update family, let $p\in(0,1]$, and let $A$ be a $p$-random subset of $\Z^2$. Then
\begin{equation}
\P_p\big([A]=\Z^2\big)=1. \tag*{\qedsymbol}
\end{equation}
\end{corollary}

This is immediate from the previous theorems: the probability that any given site is not eventually infected is zero, and there are countably many sites.

As stated above, there are essentially equivalent versions of the first two theorems stated in terms of critical probabilities. We combine these into the following theorem.

\begin{theorem}\label{th:pc}
Let $\U$ be a two-dimensional bootstrap percolation update family.
\begin{enumerate}
\item If $\U$ is critical then
\[
p_c(\Z_n^2,\U) = \left(\frac{1}{\log n}\right)^{\Theta(1)}.
\]
\item If $\U$ is supercritical then
\[
p_c(\Z_n^2,\U) = n^{-\Theta(1)}.
\]
\end{enumerate}
\end{theorem}

We shall not prove Theorem~\ref{th:pc} explicitly, but we note here that it follows from exactly the same methods as Theorems~\ref{th:main} and~\ref{th:super}.

\begin{remark}
For certain critical update families $\U$, our methods can be used to be used to determine $p_c(\Z_n^2,\U)$ up to a constant factor. This is the case, for example, with the 2-neighbour model, and thus we recover the theorem of Aizenman and Lebowitz~\cite{AL} as a special case of our results (see Theorem~\ref{th:lower} and Remark~\ref{re:bounds}). However, we do not believe that our methods can be used without further improvements to determine $p_c(\Z_n^2,\U)$ up to a constant factor for general critical models.
%The reason we do not attempt to optimize our bounds is that, in our proof of the upper bound of Theorem~\ref{th:main}, the simpler approach we use enables us to treat every critical family (and indeed every supercritical family) under the same framework. We remark further on this point in Section~\ref{se:upper}.
\end{remark}

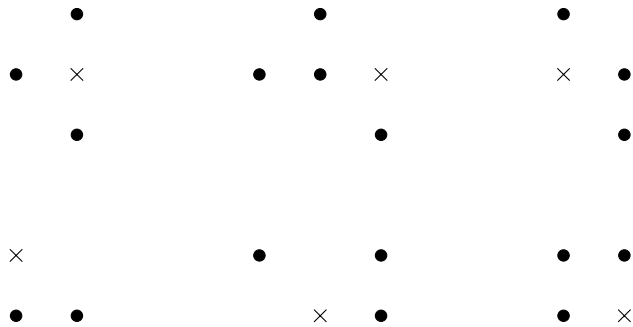
\begin{figure}[ht]
  \centering
  \begin{tikzpicture}[scale=0.8]
    \begin{scope}[every node/.style={cross out,draw,inner sep=0,minimum size=0.15cm}]
      \node at (0,1) {};
      \node at (5,0) {};
      \node at (10,0) {};
      \node at (1,4) {};
      \node at (6,4) {};
      \node at (9,4) {};
    \end{scope}
    \begin{scope}[every node/.style={circle,draw,fill,inner sep=0,minimum size=0.15cm}]
      \node at (0,0) {};
      \node at (1,0) {};
      \node at (4,1) {};
      \node at (6,0) {};
      \node at (6,1) {};
      \node at (9,0) {};
      \node at (9,1) {};
      \node at (10,1) {};
      \node at (0,4) {};
      \node at (1,3) {};
      \node at (1,5) {};
      \node at (4,4) {};
      \node at (5,4) {};
      \node at (5,5) {};
      \node at (6,3) {};
      \node at (9,5) {};
      \node at (10,3) {};
      \node at (10,4) {};
    \end{scope}
  \end{tikzpicture}
  \caption{An example of an update family $\U$ consisting of six rules. If there is a translation of one of the rules such that all of the sites marked with dots are infected, then the site marked with a cross will become infected too. It is easy to verify that $\U$ is critical. Our methods can be used to show that $p_c(\Z_n^2,\U)=\Theta(1/\log n)$, and therefore that $\U$ has the same large-scale behaviour as two-neighbour bootstrap percolation.}
  \label{fi:family}
\end{figure}

\begin{remark}
Theorem~\ref{th:pc} is a statement about critical probabilities of $\U$-bootstrap percolation on the discrete torus $\Z_n^2$. The corresponding statement with $\Z_n^2$ replaced by the grid $[n]^2$ is false. This is for trivial reasons: the infection would not in general be able to reach the corners of the grid unles $p=1-o(1)$.
\end{remark}

Returning to the first of our aims, let us now be clearer about what we mean by proving `basic results' about $\U$-bootstrap percolation: we mean establishing general properties of stable sets and the action of the processes on half-planes. This is the subject of Section \ref{se:stableset}. The main result of that section is the following classification of stable sets. Let us say that a unit vector $u\in S^1$ is \emph{rational} if it has either rational or infinite slope relative to the standard basis vectors.

\begin{theorem}\label{th:stabclass}
Let $\stab\subset S^1$. Then there exists a two-dimensional bootstrap percolation update family $\U$ of which $\stab$ is the stable set if and only if $\stab$ can be expressed as a finite union of closed intervals in $S^1$ with rational end-points.
\end{theorem}

%During the course of the proof of Theorem \ref{th:super} we show that an update family is supercritical if and only if there exist finite sets with infinite closure. Put another way, we show there exists an open semicircle $C\subset S^1$ such that $\stab\cap C=\emptyset$ if and only if there exists a finite set $A\subset\Z^2$ such that the cardinality of $[A]$ is infinite.

\subsection{Organization of the paper}

In the next section we introduce several key aspects of the proofs and a number of important definitions. Owing to the level of generality of our model, the need for some notational complexity is unfortunately unavoidable. However, we have tried to keep the burden as light as possible, and have gathered together many of the definitions for ease of reference. 

In Section~\ref{se:stableset} we prove some basic properties of stable sets, culminating in the proof of Theorem \ref{th:stabclass}. Following that, we move on to the main aim of this paper: the proof of Theorem~\ref{th:main} for critical update families. In Section \ref{se:lower}, we prove the lower bound of Theorem~\ref{th:main}. In Section~\ref{se:upperprelims} we set out the deterministic framework for the growth of `droplets' that we shall use to prove the upper bounds of both Theorem~\ref{th:main} (for critical families) and Theorem~\ref{th:super} (for supercritical families). We then complete the proofs of these theorems in Sections~\ref{se:upper} and~\ref{se:super} respectively.

Finally, in Section \ref{se:open}, we discuss open problems and conjectures, including the behaviour of subcritical families, sharpening our results, and generalizations to higher dimensions.

\section{Some aspects of the proofs}\label{se:sketch}

In this section we introduce several important definitions and ideas from the proofs of Theorems~\ref{th:main} and~\ref{th:super}. We begin with two definitions that may be familiar to the reader in the context of certain specific models (such as the 2-neighbour model or `balanced' threshold models); in each case, the definition we give here generalizes previous definitions. The first, that of a set being `internally filled', is our formal way of saying that a set becomes fully infected in $\U$-bootstrap percolation without external help.

Let us fix an update family $\U$, so that (for example) the operation of taking the closure of a set is well-defined.

\begin{definition}\label{de:ifill}
A set $X\subset\Z^2$ is \emph{internally filled by $A$} if $X\subset[X\cap A]$.
\end{definition}

In the literature, the expression `internally spanned' is frequently used to describe this property; we feel that the word `filled' better emphasizes that the whole of $X$ must be contained in the closure $[X\cap A]$.

The second definition is that of a `$\T$-droplet', which is a finite subset of $\Z^2$ used as a basis for growth.

\begin{definition}\label{de:droplet}
Let $\T\subset S^1$. A \emph{$\T$-droplet} is any finite set of the form
\[
D = \bigcap_{u\in\T} \big(\H_u + a_u\big),
\]
where $\{a_u\in\Z^2:u\in\T\}$ is a collection of sites in $\Z^2$. If $u\in\T$ and $D$ is a $\T$-droplet, then the \emph{$u$-side} of $D$ is the set
\[
\big\{ x\in D \,:\, \< y,u \> \leq \< x,u \> \text{ for all } y\in D \big\}.
\]
\end{definition}

Thus, a droplet is an intersection of half-planes. Since the stable directions for the 2-neighbour process are $(\pm 1,0)$ and $(0,\pm 1)$, if $\T$ is the set of these four unit vectors then a $\T$-droplet is just a rectangle, precisely as a `droplet' is usually taken to be in the 2-neighbour model.

Let us make two remarks about Definition~\ref{de:droplet}. First, we have defined droplets with respect to a set $\T\subset S^1$, not with respect to an update family $\U$. This is because, for a given update family $\U$, we shall need to define and work with different sets of directions to the stable set in both the upper and lower bound parts of the proof, and in each case we shall need a correspondingly different notion of a droplet.

Second, as we have just alluded to, droplets will be important in the proofs of both parts of Theorem~\ref{th:main}. From the point of view of the lower bound, they are important because they restrict growth. If $\T\subset\stab$ then the closure of a subset of a $\T$-droplet is again a subset of the same $\T$-droplet. From the point of view of the upper bound, droplets form the basis of growth. With the help of a bounded number of new sites along one of its edges, under certain circumstances a $\T$-droplet will grow into a new, larger $\T$-droplet. In fact, ensuring that these `certain circumstances' always hold is one of the central components of the proof of the upper bound of Theorem~\ref{th:main}. We return to this in Section~\ref{se:sketchupper}

Now that we have Definitions~\ref{de:ifill} and~\ref{de:droplet}, let us also make a brief remark about symmetry. One might think that the lack of symmetry should make only trivial, technical differences to the arguments, and that its assumption would cause only a little loss of generality, as was supposed in \cite{GG96,GG99}. However, nothing could be further from the truth: symmetry is the most important instrument at one's disposal in proofs of lower bounds for critical probabilities in bootstrap percolation. Before explaining why, let us be clear what we mean by `symmetry'. We mean principally symmetry of the stable set: the property that $u\in\stab$ implies $-u\in\stab$. Almost always in the past, though, authors have assumed much more that this, namely that $X\in\U$ if and only if $-X\in\U$. Why should symmetry be so useful? The main reason is that it allows one to restrict growth one-dimensionally. Suppose that $u$ and $-u$ are stable directions, and that the initial set $A\subset\Z^2$ does not intersect an infinite strip with sides orthogonal to $u$ of at least a certain width depending only on $\U$. Then $[A]$ cannot intersect that strip either. This allows tight (and easy) control of the growth of a droplet in terms only of sites near to its perimeter. It also gives an easy necessary condition for a droplet to be internally filled, which immediately implies, among other things, a certain extremal lemma (which we discuss in the next subsection). Since we do not have symmetry, we have to work much harder to obtain corresponding properties. Indeed, dealing with the lack of symmetry is the single most important challenge of the general model.

\subsection{Lower bounds for critical probabilities}

At the most basic level, our proofs of both the upper and lower bounds of Theorem~\ref{th:main} follow the corresponding proofs of Aizenman and Lebowitz~\cite{AL} for the 2-neighbour model. In~\cite{AL}, the starting point for the lower bound is to show that if $\0\in A_t$ then there exists an internally filled droplet at the scale of roughly $\log t$ within distance $t$ of the origin. Our starting point is similar, except that the property of being internally filled is too strong a property to ask for in general, so instead we use a certain notion of being `approximately internally filled'; we say a little more precisely what we mean by this shortly.
 
In order to find and bound the probability of an approximately internally filled droplet of size roughly $\log t$, we need two main ingredients. The first is a statement that says that if a droplet at a certain scale is approximately internally filled then it necessarily contains approximately internally filled droplets at all smaller scales. We call a statement of this form, or a close variant, an \emph{Aizenman--Lebowitz-type lemma}. The second ingredient is a statement that says that if a droplet $D$ of `size' $k$ is approximately internally filled, then $D\cap A$ must have cardinality at least $\delta k$, where $\delta>0$ may depend on $\U$ but not on $k$. We call a statement of this form an \emph{extremal lemma}.

In the case of the 2-neighbour model, it turns out that one can prove both of these lemmas (with `approximately' replaced by `exactly') using the same idea. (There are many ways of the proving the extremal lemma, but this is the only known way of proving the Aizenman--Lebowitz lemma. In fact, if one has a weak symmetry property, such as a pair of stable directions $u$ and $-u$, and one does not care about optimizing $\delta$, then the extremal lemma is immediate.) The idea in both cases is to use an algorithm called the \emph{rectangles process}, which builds up the closure of a finite set of sites by `merging' nearby internally filled droplets one-by-one. To be precise, the algorithm proceeds by looking for two rectangles within $\ell_1$ distance $2$ of each other, and replacing them by the smallest rectangle containing both. The two crucial properties of the algorithm are, first, that at each step, every rectangle in the current collection is internally filled, and second, that the perimeter length of the largest rectangle in the current collection does not more than double between steps.

The natural generalization of the rectangles process -- replacing `rectangles' by `$\T$-droplets' for some $\T\subset S^1$ -- is of no use. The issue is that the closure of two nearby $\T$-droplets need not be a larger $\T$-droplet, so we lose the first of the two crucial properties mentioned above, namely that the $\T$-droplets at each stage of the process be internally filled. Our solution, as mentioned above, is to weaken the requirement that the droplets be internally filled. The precise way in which we do this is to introduce a new, related algorithm, which we call the \emph{covering algorithm}. The algorithm starts by placing a $\T$-droplet (for a certain set $\T$) around as many disjoint `breakthrough blocks' in a finite initial set $K\subset\Z^2$ as possible, where `breakthrough blocks' are defined as follows.

\begin{definition}
A \emph{breakthrough block} for $X\in\U$ is any set of the form $X \cap (\H_u)^c$, or a translation by an element of $\Z^2$ of such a set, for some $u\in\stab$.
\end{definition}

The idea is that a breakthrough block is a minimal set of sites that allows at least one new site to become infected with the help of a $\T$-droplet, where $\T\subset\stab$. The covering algorithm then proceeds by merging nearby $\T$-droplets in a manner similar to the original rectangles process. One can show that if the definition of the covering algorithm we have just sketched is suitably formalized, and if the final collection of $\T$-droplets at the end of algorithm is $D_1,\dots,D_T$, then $[K]\setminus K \subset D_1\cup\dots\cup D_T$, and therefore that the covering algorithm `approximately dominates' the $\U$-bootstrap process. Crucially, one can also prove an Aizenman--Lebowitz-type lemma and an extremal lemma for so-called `covered droplets', which are simply droplets $D$ such that the set $\{D\}$ is a possible output of the covering algorithm with input $D\cap A$. Recall that these are the two main lemmas we need to prove the lower bound.

The precise definition of the covering algorithm, statements and proofs of the Aizenman--Lebowitz-type lemma and the extremal lemma, and the proof of the lower bound of Theorem~\ref{th:main} are given in Section~\ref{se:lower}.

\subsection{Upper bounds for critical probabilities}\label{se:sketchupper}

For most of this discussion we shall assume that $\U$ is critical. At the end we indicate how our methods for proving the upper bound of Theorem~\ref{th:main} for critical families can be applied to prove Theorem~\ref{th:super} for supercritical families as well.

The first issue we encounter with the upper bounds is that the stable set $\stab$ may be quite large: indeed it may contain a closed semicircle. This means that droplets may not in general grow in all directions. However, in order to show that droplets grow in at least one direction, we only need a weaker property, which follows from the definition of a critical update family and Theorem~\ref{th:stabclass}. This property is that there exists an open semicircle $C\subset S^1$ such that $|\stab\cap C|<\infty$.

We remarked after Definition \ref{de:droplet} that ensuring that droplets grow into droplets is one of the central parts of the proof of the upper bound of Theorem \ref{th:main}. What we mean by `droplets growing into droplets', or more specifically, `$\T$-droplets growing into $\T$-droplets', is that for each $\T$-droplet $D$ and each $u\in\T$, there exists a set $Z\subset\Z^2\setminus D$ of bounded size (that is, of size depending on $\U$ and $\T$ but not on $D$) such that $[D\cup Z]$ contains an entire new `row' of sites on the $u$-side, extending all the way to the corners of the $u$-side. Thus, $[D\cup Z]$ is itself a $\T$-droplet. Once we have a statement of this form, then our proof of the upper bound of Theorem~\ref{th:main} follows by methods similar to those used by van Enter and Hulshof~\cite{vEH} to determine the critical probability of a particular so-called `unbalanced' model.

The growth property that we have just described itself hinges on two further properties. The first is that if $u\in S^1$ is such that there exists an open interval in $I\subset S^1$ with $u\in I$ but such that $\stab\cap I\setminus\{u\}=\emptyset$, then there exists a finite set $Z\subset\ell_u$ such that $\ell_u\subset [\H_u\cup Z]$, where
\[
\ell_u := \big\{ x\in\Z^2 \,:\, \< x,u \> = 0 \}.
\]
(Of course, if $u\notin\stab$ then $Z=\emptyset$ will do here.) This property, which is Lemma~\ref{le:ublock}, will allow us to grow to within a constant distance of the corners of the $\T$-droplet. However, if we choose our set $\T$ to be such that $\T\subset\stab$ -- as in all previous proofs of this type of theorem -- then the new row of infected sites \emph{does not} in general grow all the way to the corners. The second property that we need in order to rectify this problem is that there exists a finite set $\quasi\subset S^1$, which we call the set of \emph{quasi-stable directions}, such that if $\T$ is the union of $\quasi$ and a certain subset of $\stab$, then this corners property holds. This is the content of Lemma~\ref{le:quasi}.

Let us briefly observe that both of the properties of the previous paragraph trivially hold in the case of the 2-neighbour model: for the first, we may always take $Z=\{\0\}$, and for the second, we may take $\quasi=\emptyset$. This again serves to highlight that many of the issues that arise in connection with the general model do not appear at all in connection with the 2-neighbour model.

In Section~\ref{se:upperprelims} we establish the key growth properties described above. We then apply them in Section~\ref{se:upper} to prove the upper bound of Theorem~\ref{th:main} for critical families. In Section~\ref{se:super}, we show that the framework we have set up for critical families may be adapted to prove Theorem~\ref{th:super} for supercritical families, by observing that all of the sets $Z$ mentioned above may be taken to be empty if $\U$ is supercritical and the open semicircle $C\subset S^1$ is chosen so that $\stab\cap C=\emptyset$.

\subsection{Definitions and conventions}\label{se:defs}

Let us collect a few of the most important definitions and conventions that we shall use throughout the paper. We begin with three conventions:
\begin{itemize}
\item The set $A$ is always a $p$-random subset of $\Z^2$, unless explicitly stated otherwise.
\item The update family $\U$ is fixed throughout. Constants will be allowed to depend on $\U$ (and hence on $\stab$, etc.), but not on $p$.
\item The norm $\|\cdot\|$ denotes the Euclidean norm $\|\cdot\|_2$.
\end{itemize}

Moving on to definitions, we begin with certain subsets of the plane. Recall that we have already defined the half-plane $\H_u=\{x\in\Z^2:\<x,u\> <0\}$ and the line $\ell_u=\{x\in\Z^2:\<x,u\>=0\}$. We shall also need the shifted half-plane:
\[
\H_u(a) := \big\{ x\in\Z^2 \,:\, \<x-a,u\> < 0 \big \},
\]
where $a\in\R^2$. If $a\in\Z^2$ then we have $\H_u(a)=\H_u+a$, but this is not true in general. Next, let $u\in S^1$ be rational (recall that this means that $u$ has either rational or infinite slope with respect to the standard basis vectors of $\R^2$) and consider the lines orthogonal to $u$ that intersect $\Z^2$. These lines are parallel and discrete, and there is a natural integer-valued indexing for them. Let $\ell_u(0)$ be the line intersecting the origin, and for each $i\in\Z$, let $\ell_u(i)$ be the $i$th line in the direction of $u$.

Next, let $\TT:=\R/2\pi\Z$. We shall often need to change between elements of $S^1$ and elements of $\TT$, so we define the following natural bijection. Let $u:\TT\to S^1$ be the function
\[
u(\theta) := (\cos\theta,\sin\theta),
\]
and let $\theta:S^1\to\TT$ be the inverse function, so that if $u=u(\theta)$ then $\theta(u)=\theta$. We extend the domain of the function $\theta$ to the whole of $\R^2\setminus\{\0\}$ in the obvious way: if $x\in\R^2\setminus\{\0\}$ then we define $\theta(x)$ to be equal to $\theta(x/\|x\|)$. We also use standard function notation, so that for example $u(T):=\{u(\theta):\theta\in T\}$ for a subset $T\subset\TT$, and we abuse notation slightly by abbreviating $u\big((\theta_1,\theta_2)\big)$ to $u(\theta_1,\theta_2)$.

Now let $\T\subset S^1$ and let $u,v\in\T$. We shall say that $u$ and $v$ are \emph{consecutive} in $\T$ if $\T \cap u\big(\theta(u),\theta(v)\big) = \emptyset$. (Note that this is not a symmetric relation: the statements that $u$ and $v$ are consecutive in $\T$, and that $v$ and $u$ are consecutive in $\T$, are not equivalent.)

For each update rule $X$, let
\[
T(X) := \{u\in S^1 \, : \, X \subset \H_u\},
\]
and let
\[
\Theta(X) := \{\theta\in\TT \, : \, u(\theta)\in T(X) \} = \{\theta\in\TT \, : \, X \subset \H_{u(\theta)}\}.
\]
We think of $T(X)$ as the set of directions \emph{destabilized} by the update rule $X$. If $u\in\stab^c$ then there must exist a rule $X\in\U$ such that $X\subset\H_u$, which implies that
\begin{equation}\label{eq:Sc}
\stab^c = \bigcup_{X\in\U} T(X).
\end{equation}

Finally, we shall need the following measure of the \emph{diameter} of a finite set $K\subset\Z^2$:
\[
\diam(K) := \max \big\{ \|x-y\| \,:\, x,y\in K \big\}.
\]

\section{The structure of stable sets}\label{se:stableset}

In this section we derive a number of basic properties of the stable set and introduce some important definitions relating to the geometry of the stable set. Our main aim is to prove Theorem~\ref{th:stabclass}, the classification of stable sets.

We begin with a simple lemma that establishes the dichotomy $[\H_u]\in\big\{\H_u,\Z^2\}$ for every $u\in S^1$.

\begin{lemma}
Let $u\in S^1$ be an unstable direction for $\U$. Then $[\H_u]=\Z^2$.
\end{lemma}

\begin{proof}
Let $u\in S^1$ be unstable and let $a\in\Z^2\setminus\H_u$ and $X\in\U$ be such that $a+X\subset\H_u$. Then $y+X\subset\H_u$ for all $y\in\ell_u$, since if $x\in X$ and $y\in\ell_u$ then
\[
\< y+x,u \> = \< x,u \> \leq \< a+x,u \> < 0.
\]
Thus $\ell_u\subset A_1\setminus A_0$.

If $u$ is rational then we are done: $\ell_u(j)\subset A_{j+1}$ for all integers $j\geq 0$, so $[\H_u]=\Z^2$. If $u$ is irrational, on the other hand, then all we know so far is that $\0\in A_1\setminus A_0$. We claim that there exists a site $b\in\Z^2$ such that $\< b,u \> >0$ and $\H_u(jb)\subset A_j$ for all $j\in\N$. Let
\[
\delta := \sup \big\{ \lambda\in\R \, : \, X \subset \H_u(-\lambda u) \big\} = \min \big\{ -\<x,u\> \,:\, x\in X \big\}.
\]
If $y\in\H_u(\delta u)$ then $y+X\subset\H_u$, so we have $\H_u(\delta u)\subset A_1$. Also, the set $X$ is finite and contained in $\H_u$, so $\delta>0$. Thus, since $u$ is irrational, there exists a site $b\in\H_u(\delta u)\setminus\H_u$, and we have $\H_u(b)\subset\H_u(\delta u)\subset A_1$. Now $\H_u(b)$ has a site on its boundary, namely $b$, so it is congruent to $\H_u$, and therefore $\H_u(jb)\subset A_j$ for all $j\in\N$, as claimed.
\end{proof}

The next lemma, together with the observation in \eqref{eq:Sc}, is one of the two implications in Theorem~\ref{th:stabclass}.

\begin{lemma}\label{le:Tclass}
Let $X$ be an update rule. Then $T(X)$ is either empty or an open interval in $S^1$ with rational end-points.
\end{lemma}

\begin{proof}
For each site $x\in\Z^2\setminus\{\0\}$ define the set $T(x) = \{u:\< x,u\> < 0\}$. It is easy to see that
\begin{equation}\label{eq:Tx}
T(x) = u\big(\theta(x)-\pi/2,\theta(x)+\pi/2\big).
\end{equation}
Furthermore, since $u\big(\theta(x)\big)=x/\|x\|$ is rational, so too are $u\big(\theta(x)-\pi/2\big)$ and $u\big(\theta(x)+\pi/2\big)$. Thus, $T(x)$ is an open interval with rational end-points.

Now simply notice that
\begin{equation}\label{eq:TX}
T(X) = \bigcap_{x\in X} T(x),
\end{equation}
so $T(X)$ is a finite intersection of open intervals with rational end-points, each interval of length exactly $\pi$. The assertion of the lemma follows.
\end{proof}

\begin{remark}
Using identities for $\stab^c$ in \eqref{eq:Sc}, for $T(X)$ in \eqref{eq:TX}, and for $T(x)$ in \eqref{eq:Tx}, we have
\begin{equation}\label{eq:findstab}
\stab^c = \bigcup_{X\in\U} \bigcap_{x\in X} u\big(\theta(x)-\pi/2,\theta(x)+\pi/2\big).
\end{equation}
From this it follows immediately that there exists a polynomial time (in the sum of the cardinalities of the update rules, say) algorithm for computing the stable set $\stab$ of an update family $\U$.
\end{remark}

We are now able to complete the proof of Theorem \ref{th:stabclass}. Recall that the theorem is a classification of stable sets: it says that a subset of the circle is the stable set of some update family if and only if it can be expressed as a finite union of closed intervals with rational end-points.

\begin{proof}[Proof of Theorem \ref{th:stabclass}]
One implication -- that stable sets are finite unions of closed intervals with rational end-points -- is a consequence of Lemma \ref{le:Tclass} and \eqref{eq:Sc}. Therefore our task is just to show the converse. To that end, let $\stab\subset S^1$ be such that its complement may be written in the form
\[
\stab^c = \bigcup_{i=1}^m T_i,
\]
where
\[
T_i = u(\theta_i^-,\theta_i^+)
\]
for $i=1,\dots,m$, and the $T_i$ are disjoint and have rational end-points. Since $u(\theta_i^-)$ and $u(\theta_i^+)$ are rational, for each $i$, it follows that there exists a non-zero site $x_i^-\in\ell_{u(\theta_i^-)}$ to the right of the origin as one looks in the direction of $u(\theta_i^-)$, and a non-zero site $x_i^+\in\ell_{u(\theta_i^+)}$ to the left of the origin as one looks in the direction of $u(\theta_i^+)$. Let
\[
X_i := \{x_i^-,x_i^+\}
\]
be update rules for each $i=1,\dots,m$, and let $\U:=\{X_1,\dots,X_m\}$ be an update family and $\stab'=\stab'(\U)$ its stable set. We claim that $\stab'=\stab$.

\begin{figure}[ht]
  \centering
  \begin{tikzpicture}[>=latex] 
    \draw (170:1) -- (-10:5);
    \draw (60:1) -- (-120:5);
    \draw[->] (-10:2.5) -- ++(80:1);
    \path (-10:2.5) ++(80:1.3) node {$u(\theta_i^-)$};
    \draw[->] (-120:2.5) -- ++(150:1);
    \path (-120:2.5) ++(150:1.6) node {$u(\theta_i^+)$};
    \node[circle,fill,inner sep=0,minimum size=0.08cm] at (0,0) {};
    \node at (-0.1,0.3) {$\0$};
    \node[circle,fill,inner sep=0,minimum size=0.15cm,label=-100:$x_i^-$] at (-10:4) {};
    \node[circle,fill,inner sep=0,minimum size=0.15cm,label=-30:$x_i^+$] at (-120:1.8) {};
  \end{tikzpicture}
  \caption{The rule $X_i$ consists of the two sites $x_i^-$ and $x_i^+$. It destabilizes the interval $(\theta_i^-,\theta_i^+)$.}
  \label{fi:stabclass}
\end{figure}
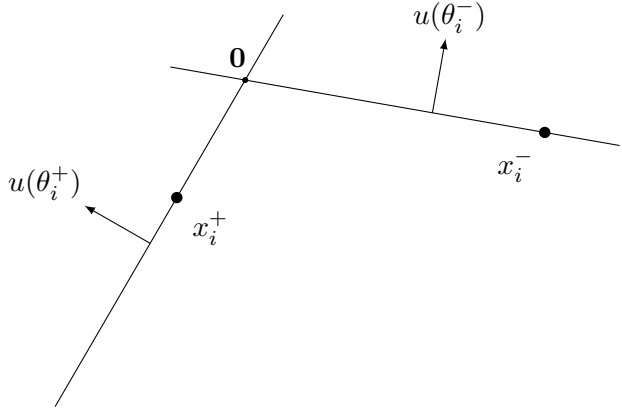

By \eqref{eq:Sc},
\[
(\stab')^c = \bigcup_{i=1}^m T(X_i),
\]
so it is enough to show that $T(X_i)=T_i$ for every $i$. But
\begin{align*}
T(X_i) &= \{u:\< x_i^-,u\> < 0\} \cap \{u:\< x_i^+,u\> < 0\} \\
&= u(\theta_i^-,\theta_i^-+\pi) \cap u(\theta_i^+-\pi,\theta_i^+) \\
&= u(\theta_i^-,\theta_i^+),
\end{align*}
which is the definition of $T_i$.
\end{proof}

\section{The lower bound for critical families}\label{se:lower}

The next three sections are the primary focus of the paper. We study critical\footnote{In Section~\ref{se:upperprelims} we only assume that $\U$ is critical or supercritical.} $\U$-bootstrap percolation with the aim of proving Theorem~\ref{th:main}: the lower bound in this section and the upper bound in Section~\ref{se:upper}. Recall that an update family is critical if the following two properties of its stable set $\stab$ hold. First, $\stab\cap C$ is non-empty for every open semicircle $C\subset S^1$, and second, there exists an open semicircle $C\subset S^1$ not containing any strongly stable directions, where a stable direction $u$ is said to be strongly stable if it is contained in an open interval of stable directions.

%In this section we shall prove the lower bound of Theorem~\ref{th:main} with the following value of $\alpha_1$. 
Recall that a breakthrough block for $X\in\U$ is any set of the form $X\cap(\H_u)^c$, where $u\in\stab$. (It is worth observing that there are no empty breakthrough blocks, because the direction $u$ is assumed to be stable.) Let $\mathcal{B}$ be the collection of sets $B\subset\Z^2$ such that $B$ is a breakthrough block for some $X\in\U$, and let
\begin{equation}\label{eq:alpha1}
\alpha_1 = \alpha_1(\U) := \min \big\{ |B| \,:\, B\in\mathcal{B} \big\}.
\end{equation}
The lower bound of Theorem~\ref{th:main} follows from the following theorem. We remind the reader that the stopping time $\tau$ is the minimal $t\geq 0$ such that $\0\in A_t$.
%In light of the classification of possible stable sets given by Theorem \ref{th:stabclass}, the second condition is equivalent to the statement that there exists an open semicircle $C\subset S^1$ such that $\stab\cap C$ has finite cardinality.

\begin{theorem}\label{th:lower}
Let $\U$ be a critical update family and let
\[
p \leq \left(\frac{\lambda}{\log t}\right)^{1/\alpha_1},
\]
where $\lambda>0$ is sufficiently small. Then $\tau\geq t$ with high probability as $t\to\infty$.
\end{theorem}

It follows from the definition of a critical update family that if $\U$ is critical then there exists a set $\stabl=\stabl(\U)$ of 3 or 4 stable directions such that $\0$ lies in the interior of the convex hull of $\stabl$. (Equivalently, there exist (finite) $\stabl$-droplets.) Throughout this section we fix the critical update family $\U$ and the set $\stabl$ of 3 or 4 stable directions, and we shall use the convention that all droplets are $\stabl$-droplets.

In Subsection~\ref{se:cover} we define the `covering algorithm', which will be our approximation to the rectangles process of the 2-neighbour model, and we define what it means for a droplet to be `covered', which will be our notion of being `approximately internally filled'. We use these definitions to derive an Aizenman--Lebowitz-type lemma (Lemma~\ref{le:AL}) and an extremal lemma (Lemma~\ref{le:extremal}) for `covered' droplets. In Subsection~\ref{se:lowerproof} we assemble the tools of the previous subsection to complete the proof of Theorem~\ref{th:lower}.

\subsection{Covered droplets}\label{se:cover}

We begin by defining the `covering algorithm', which takes as its input a finite set $K\subset\Z^2$ of sites and returns a collection $\D$ of $\stabl$-droplets that approximately cover $[K]$, in a sense made precise in Lemma~\ref{le:coverclosed}.

Let $\hat{D}$ be a fixed $\stabl$-droplet such that $X\subset\hat{D}$ for all $X\in\U$. Given a finite set $K\subset\Z^2$, let $D(K)$ be the minimal $\stabl$-droplet containing $K$.

\begin{definition}\label{de:covering} \emph{(Covering algorithm.)}
Let $\U$ be critical and let $K\subset\Z^2$ be finite. Let $B_1,\dots,B_{k_0}$ be a maximal collection of disjoint breakthrough blocks in $K$, and let $\D^0:=\{D_1^0,\dots,D_{k_0}^0\}$ be a collection of copies of $\hat{D}$ such that $B_i\subset D_i^0$ for each $i=1,\dots,k_0$. This is step 0. Now repeat the following procedure for each $t\geq 0$ until STOP. At the start of step $t+1$ of the algorithm, there is a collection $\D^t=\{D_1^t,\dots,D_{k_t}^t\}$ of droplets, where $k_t=k_0-t$. If there do not exist indices $i$ and $j$ and a site $x\in\Z^2$ such that
\begin{equation}\label{eq:dropletsmerge}
D_i^t \cap (\hat{D}+x) \neq \emptyset \qquad \text{and} \qquad D_j^t \cap (\hat{D}+x) \neq \emptyset
\end{equation}
then STOP. If there do exist such $i$, $j$ and $x$, then choose one such triple. Construct $\D^{t+1}$ from $\D^t$ by deleting from it $D_i^t$ and $D_j^t$ and adding to it $D(D_i^t\cup D_j^t)$:
\[
\D^{t+1} := \big(\D^t \setminus \{D_i^t, D_j^t\} \big) \cup \{D(D_i^t\cup D_j^t)\}.
\]
This is the end of step $t+1$. The output of the algorithm is the set $\D^T:=\{D_1^T,\dots,D_k^T\}$, where $k=k_0-T$, and $T$ is the last $t$ before STOP.
\end{definition}

\begin{remark}
We could instead have defined the covering algorithm by taking our initial collection of droplets $\D^0$ to consist of a copy of $\hat{D}$ around every \emph{site} of $K$, rather than around every element of a maximal collection of disjoint breakthrough blocks. The effect of this would have been that we would have obtained $\alpha_1=1$ as the bound in Theorem~\ref{th:lower} for every critical update family. We use the present method because in some cases it gives better bounds, and because it may be useful in potential future refinements of our results.
\end{remark}

We note that the covering algorithm is commutative in the sense that the order in which droplets are combined does not affect the output of the algorithm. We omit the proof of this property, since we shall not actually need to use it. On the other hand, the output of the covering algorithm \emph{does} depend on the particular choice of disjoint breakthrough blocks $B_1,\dots,B_{k_0}$, and on the particular translates of $\hat{D}$ we choose to cover those breakthrough blocks.

\begin{definition}
We say that $\D$ is a \emph{cover} of a finite set $K\subset\Z^2$ if it is a possible output of the covering algorithm with inputs $\U$ and $K$. An $\stabl$-droplet $D$ is \emph{covered} if $\D=\{D\}$ for some cover $\D$ of $D\cap A$.
\end{definition}

The following lemma allows us to bound the closure of a finite set in terms of an arbitrary output of the covering algorithm applied to that set; it is the sense in which the covering algorithm approximately dominates the bootstrap process.

\begin{lemma}\label{le:coverclosed}
Let $\D=\{D_1^T,\dots,D_k^T\}$ be an output of the covering algorithm applied to a finite set $K\subset\Z^2$. Then
\[
[K]\setminus K \subset D_1^T\cup\dots\cup D_k^T.
\]
\end{lemma}

\begin{proof}
Let $L:=D_1^T\cup\dots\cup D_k^T$. We claim that $[L\cup K]=L\cup K$. In order to prove this, we must show that there do not exist $X\in\U$ and $x\in\Z^2$ such that $x+X\subset L\cup K$ and $x\notin L\cup K$.

First, it is clear that we cannot have $x+X\subset K\setminus L$. This is simply because $x+X$ contains a breakthrough block, so $x+X\subset K\setminus L$ would contradict the maximality of the initial collection $B_1,\dots,B_{k_0}$ of breakthrough blocks in the covering algorithm.

Second, we cannot have $x+X\subset L\cup K$ and $(x+X)\cap L\neq\emptyset$. To see this, first note that $x+X$ cannot intersect more than one of the $D_i^T$, because if $t=T$ then \eqref{eq:dropletsmerge} does not hold for any $1\leq i<j\leq k$ and $\hat{D}$ contains every update rule. Also, $x+X$ cannot be contained in $D_i^T$ for any $i$, because each $D_i^T$ is closed. Therefore, if $(x+X)\cap L\neq\emptyset$, then there exists a unique $i$ such that $(x+X)\cap D_i^T\neq\emptyset$ and $(x+X)\setminus D_i^T\subset K\setminus L$. Hence also, there exists $u\in\stabl$ and $y\in\Z^2$ such that
\[
\emptyset \neq (x+X)\cap\H_u(y)^c \subset K\setminus L,
\]
again contradicting the maximality of the initial collection of breakthrough blocks.

%Finally, we cannot have $x+X\subset L\cup K$, $(x+X)\cap L\neq\emptyset$ and $(x+X)\cap K\neq\emptyset$. To see this, first note that if $(x+X)\cap L\neq\emptyset$ then there is exactly one value of $i$ such that $(x+X)\cap D_i^T\neq\emptyset$ (this is again because \eqref{eq:dropletsmerge} does not hold). Therefore there exists $u\in\stabl$ and $y\in\Z^2\setminus L$ such that $(x+X)\cap\H_u(y)^c\subset K$, contradicting the maximality of the initial collection of breakthrough blocks in the covering algorithm.

Thus we have $[L\cup K]=L\cup K$ as claimed, and therefore
\[
[K] \subset [L\cup K] = L\cup K,
\]
which proves the lemma.
\end{proof}

We require one further ingredient for the extremal lemma, which is a subadditivity lemma for the diameters of intersecting droplets. We omit the proof of the lemma because our droplets are always either triangles or parallelograms, and in these cases the lemma is a triviality. However, we note for reference that the lemma holds in much greater generality than is stated here: it holds for droplets with respect to an arbitrary (fixed) finite subset of $S^1$. The interested reader may refer to Lemma~23 of version 2 of the present paper on the arXiv for a full proof.\footnote{Version 2 of the present paper, which has the alternative name `Neighbourhood family percolation', may be found at \texttt{http://arxiv.org/abs/1204.3980v2}.}

\begin{lemma}\label{le:subadd}
Let $D_1$ and $D_2$ be $\stabl$-droplets such that $D_1\cap D_2\neq\emptyset$. Then
\begin{equation}
\diam\big( D(D_1 \cup D_2) \big) \leq \diam(D_1) + \diam(D_2). \tag*{\qed}
\end{equation}
\end{lemma}

We are now ready to prove the extremal lemma for covered droplets.

\begin{lemma}\label{le:extremal} \emph{(Extremal lemma.)}
Let $\U$ be critical and let $D$ be a covered droplet. Then $D\cap A$ contains at least $\Omega\big(\diam(D)\big)$ disjoint breakthrough blocks.
\end{lemma}

\begin{proof}
Apply the covering algorithm with input $D\cap A$. The algorithm starts with $k_0$ droplets containing disjoint breakthrough blocks from $D\cap A$, and it terminates after $k_0$ steps with a single droplet. At the $(t+1)$th step of the algorithm, we replace droplets $D_i^t$ and $D_j^t$ by $D(D_i^t\cup D_j^t)$, for some $i$ and $j$, where $D_i^t$ and $D_j^t$ are such that \eqref{eq:dropletsmerge} holds. It follows by Lemma~\ref{le:subadd} that
\begin{equation}\label{eq:dropletsmerge2}
\diam\big( D(D_i^t \cup D_j^t) \big) \leq \diam(D_i^t) + \diam(D_j^t) + \diam(\hat{D}).
\end{equation}
Thus, since $\diam(\hat{D})=O(1)$, the quantity
\[
\sum_{D_i^t\in\D} \diam(D_i^t)
\]
increases by at most $O(1)$ at each step of the algorithm, and hence,
\[
\diam(D) \leq k_0 \diam(\hat{D}) + O(k_0).
\]
It follows that $k_0=\Omega\big(\diam(D)\big)$, as required.
\end{proof}

Finally in this subsection, we prove the Aizenman--Lebowitz-type lemma for covered droplets.

\begin{lemma}\label{le:AL} \emph{(Aizenman--Lebowitz-type lemma.)}
Let $D$ be a covered droplet. Then for every $1\leq k\leq \diam(D)$ there exists a covered droplet $D'\subset D$ such that $k\leq\diam(D')\leq 3k$.
\end{lemma}

\begin{proof}
The proof is similar to (but even simpler than) that of Lemma~\ref{le:extremal}. Apply the covering algorithm with input $D\cap A$. By \eqref{eq:dropletsmerge2}, the quantity
\[
\max \big\{ \diam(D_i^t) \,:\, D_i^t\in\D^t \big\}
\]
at most triples at each step of the algorithm, and moreover, it is easy to see that every droplet $D_i^t\in\D^t$ is covered, which proves the lemma.
\end{proof}

\subsection{Proof of Theorem \ref{th:lower}}\label{se:lowerproof}

We require two calculations for the proof of Theorem \ref{th:lower}. The first says that the probability of there existing a droplet of diameter roughly $\log t$ within distance $O(t)$ of the origin is small. The second says that the probability of there existing a strongly covered droplet containing the origin with diameter at most $\log t$ is also small. These calculations are Lemmas \ref{le:lbcalc1} and \ref{le:lbcalc2} respectively. The proof of Theorem \ref{th:lower} will split naturally into these two cases.

If $k\in\N$ then let $D(k)$ be the minimal droplet containing the origin all of whose sides have $\ell_2$ distance at least $\lambda k$ from the origin, where $\lambda$ is a large constant. Note that we have previously defined $D(K)$ for $K$ a subset of $\Z^2$; we hope this slight abuse of notation does not cause any confusion.

We begin with the first calculation, which says that it is unlikely that there is a covered droplet of diameter roughly $\log t$ contained in $D(t)$.

\begin{lemma}\label{le:lbcalc1}
Let $\U$ be critical, let $\epsilon>0$ be sufficiently small, and let
\[
p \leq \left(\frac{\epsilon}{\log t}\right)^{1/\alpha_1}.
\]
Then the probability that there exists a covered droplet $D\subset D(t)$ such that $\log t\leq \diam(D)\leq 3\log t$ is $o(1)$ as $t\to\infty$.
\end{lemma}

\begin{proof}
By Lemma~\ref{le:extremal}, there exists a constant $\delta>0$ such that if $D$ is a covered droplet with $\log t\leq \diam(D)\leq 3\log t$ then $D\cap A$ contains at least $\delta \log t$ breakthrough blocks. Since there are $O(1)$ distinct breakthrough blocks and each breakthrough block has size at least $\alpha_1$, the probability that a particular such droplet is covered is at most
\[
\binom{O(\log t)^2}{\delta\log t} p^{\alpha_1 \delta \log t} \leq \big( O(p^{\alpha_1} \log t) \big)^{\delta \log t} \leq \frac{1}{t^3},
\]
because we chose $\epsilon$ sufficiently small. Now, the number of droplets $D\subset D(t)$ with $\log t\leq \diam(D)\leq 3\log t$ is $t^2(\log t)^{O(1)}$. Therefore, the probability that there exists a covered droplet $D\subset D(t)$ with $\log t\leq \diam(D)\leq 3\log t$ is at most
\[
t^2 (\log t)^{O(1)} \cdot \frac{1}{t^3} = o(1),
\]
as required.
\end{proof}

The next calculation says that it is unlikely that there is a covered droplet contained in $D(\log t)$ that itself contains the origin.

\begin{lemma}\label{le:lbcalc2}
Let $\U$ be critical, let $\epsilon>0$ be sufficiently small, and let
\[
p \leq \left(\frac{\epsilon}{\log t}\right)^{1/\alpha_1}.
\]
Then the probability that the origin is contained in a covered droplet $D\subset D(\log t)$ is $o(1)$ as $t\to\infty$.
\end{lemma}

\begin{proof}
We begin as in the previous lemma: by Lemma~\ref{le:extremal} there exists $\delta>0$ such that if $D$ is a covered droplet with diameter at least $k$ then $D\cap A$ contains at least $\delta k$ disjoint breakthrough blocks. Thus the probability that a droplet with diameter between $k$ and $k+1$ is covered is at most
\[
\binom{O(k^2)}{\delta k} p^{\alpha_1 \delta k} \leq \big( O(p^{\alpha_1}k) \big)^{\delta k} = \left(\frac{O(\epsilon k)}{\log t}\right)^{\delta k}.
\]
The number of droplets of diameter between $k$ and $k+1$ that contain the origin is at most $k^{O(1)}$. Hence, the probability that there exists a covered droplet $D\subset D(\log t)$ such that $\0\in D$ is at most
\[
\sum_{k=1}^{O(\log t)} k^{O(1)} \left(\frac{O(\epsilon k)}{\log t}\right)^{\delta k}.
\]
By breaking up the sum at $k=(\log\log t)^2$, this in turn is at most
\[
(\log\log t)^{O(1)} \left(\frac{(\log\log t)^2}{\log t}\right)^\delta + O(\log t)^{O(1)} O(\epsilon)^{\delta (\log\log t)^2} = o(1),
\]
which completes the proof of the lemma.
\end{proof}

We are now ready to prove Theorem \ref{th:lower}.

\begin{proof}[Proof of Theorem \ref{th:lower}]
Let $p$ be such that
\[
p \leq \left(\frac{\epsilon}{\log t}\right)^{1/\alpha_1},
\]
where $\epsilon>0$ is sufficiently small. We shall show that $\tau\geq t$ with high probability as $t\to\infty$.

First we observe that if $\0\in A_t$ then we must have $\0\in[D(t)\cap A]$, provided the constant $\lambda$ in the definition of $D(t)$ is sufficiently large. Therefore, by Lemma~\ref{le:coverclosed}, either $\0\in A$ or $\0\in D$, where $D\subset D(t)$ is a covered droplet. Since the probability of the former event is $p=o(1)$, we need only bound the probability of the latter.

We split into two cases according to the size and position of $D$. Suppose first that $D\subset D(\log t)$. Then the probability that $D$ is covered is $o(1)$ by Lemma~\ref{le:lbcalc2}. On the other hand, if $D\subset D(t)$ but $D\not\subset D(\log t)$, then by Lemma~\ref{le:AL} there exists a covered droplet $D'\subset D$ such that $\log t\leq \diam(D')\leq 3\log t$. But the probability of this event is also $o(1)$, by Lemma~\ref{le:lbcalc1}, and this completes the proof of the theorem.
\end{proof}

\section{Upper bounds: the growth of droplets}\label{se:upperprelims}

The aim of this section is to develop a deterministic framework for growth in $\U$-bootstrap percolation. In the subsequent two sections, we shall use this framework to bound from below the growth that occurs under the action of a critical (in Section~\ref{se:upper}) or supercritical (in Section~\ref{se:super}) update family. Our aim in constructing this framework can be divided into three sub-aims. The first is to establish a means by which an entire half-plane $\H_u$ can grow by a single row; this we do by introducing the notion of `$u$-blocks' in Subsection~\ref{se:newline}. The second is to transfer from growth of half-planes by a single row to growth of droplets by a single row; this we do by introducing the notion of `quasi-stability' in Subsection~\ref{se:quasi}. The third and final sub-aim is to define sequences of nested droplets and give sufficient conditions for growth to occur between successive droplets; this we do in Subsection~\ref{se:growth}.

\subsection{The infection of new lines}\label{se:newline}

Our first task, then, is to show that there exists a finite set $Z\subset\Z^2$ such that $\ell_u\subset[\H_u\cup Z]$, provided $u\in S^1$ satisfies certain natural conditions. We shall in fact show that one can take $Z\subset\ell_u$.

Given a stable direction $u$, we define $\ell_u^l$ to be the set of sites in $\ell_u$ that lie to the left of the origin as one looks in the direction of $u$, and $\ell_u^r$ to be the set of sites in $\ell_u$ that lie to the right of the origin as one looks in the direction of $u$.

\begin{definition}
Let $u\in S^1$. A finite set $Z\subset\Z^2$ is a \emph{$u$-left-block} (respectively, a \emph{$u$-right-block}) if there exists $x\in\Z^2$ such that $x+Z$ is a subset of $\ell_u$ and consists of consecutive sites in $\ell_u$, and if $\ell_u^l\subset\big[\H_u\cup(x+Z)\big]$ (respectively, $\ell_u^r\subset\big[\H_u\cup(x+Z)\big]$). It is a \emph{$u$-block} if it is both a $u$-left-block and a $u$-right-block.
\end{definition}

These special breakthrough blocks are so-called because we think of them as being the means of growing left or right along the edge of a half-plane or along a side of a droplet. Note that the empty set is a $u$-block (and also a $u$-left-block and a $u$-right-block) if and only if $u$ is unstable.

A stable direction $u(\theta)$ is said to be \emph{left-isolated} if there exists $\delta>0$ such that for all $\phi\in(\theta,\theta+\delta)$, the direction $u(\phi)$ is unstable. Similarly, $u(\theta)$ is said to be \emph{right-isolated} if there exists $\delta>0$ such that for all $\phi\in(\theta-\delta,\theta)$, the direction $u(\phi)$ is unstable. We say that $u$ is \emph{isolated} if it is both left-isolated and right-isolated.

The following lemma should be thought of as saying that there exist $u$-blocks for every $u\in\stab$ that will be relevant to us.

\begin{lemma}\label{le:ublock}
There exists a $u$-left-block for every left-isolated stable direction $u$, a $u$-right-block for every right-isolated stable direction $u$, and a $u$-block for every isolated stable direction $u$.
\end{lemma}

\begin{proof}
Let $u=u(\theta)$ be a left-isolated stable direction. Then there exists an update rule $X$ and $\delta>0$ such that $u(\phi)$ is destabilized by $X$ for every $\phi$ in the open interval $I = (\theta,\theta+\delta)$. Thus
\[
X \subset \bigcap_{\phi\in I} \H_{u(\phi)} \subset \H_u \cup \ell_u^l,
\]
so $X\cap(\H_u)^c$ is a $u$-left-block. If $u$ is right-isolated then a $u$-right-block exists by symmetry. If $u$ is isolated, then it is left- and right-isolated, so the result follows.
\end{proof}

%Given a breakthrough block $B$, let $\|B\|$ be the \emph{length} of $B$, defined by
%\[
%\|B\| := \max \, \{\|x\| \, : \, x\in B\}.
%\]
%For each stable direction $u$, let $\mathcal{B}_u^l$ be the set of all $u$-left-blocks. Let $B_u^l=X_u^l\cap\H_u$ be a $u$-left-block with minimal length, so if $B\in\mathcal{B}_u^l$ then $\|B_u^l\| \leq \|B\|$. We call $B_u^r$ \emph{the $u$-left-block} and $X_u^l$ \emph{the $u$-left-rule}. Similarly, let $\mathcal{B}_u^r$ be the set of all $u$-right-blocks and let $B_u^r=X_u^r\cap\H_u$ be a $u$-right-block with minimal length, so if $B\in\mathcal{B}_u^r$ then $\|B_u^r\| \leq \|B\|$. We call $B_u^r$ \emph{the $u$-right-block} and $X_u^r$ \emph{the $u$-right-rule}.

\subsection{Quasi-stability}\label{se:quasi}

Our second task of this section is to transfer the results of the previous subsection from half-planes to droplets: that is, to establish conditions under which a $u$-block can be used to extend a droplet by a single row along its $u$-side. The important property that we need is that the infection grows \emph{all the way to the corners}, not just to within a constant distance.

In order to illustrate why this is not straightforward, let us give an example. Suppose $X_1$ and $X_2$ are two update rules such that $\Theta(X_i) = (\phi_i,\psi_i)$ for $i=1,2$, and $\phi_2 < \phi_1 < \psi_2 < \psi_1$ (see Figure~\ref{fi:quasi}). Thus, there are no stable directions between $u(\phi_2)$ and $u(\psi_1)$. Suppose (although it is not necessary), that in fact $u(\phi_2)$ and $u(\psi_1)$ are stable directions, so $\stab$-droplets have consecutive sides with outward normals in the directions of $u(\phi_2)$ and $u(\psi_1)$. Now suppose we are growing leftwards along the $u(\phi_2)$-side, and we are nearly at the corner of the droplet. How are we going to reach the corner? $X_1$ does not fit inside the droplet near the corner because $\phi_1>\phi_2$, so it partly lies outside the $u(\phi_2)$-side. $X_2$ does not fit inside the droplet near the corner either, because $\psi_2<\psi_1$, so it partly lies outside the $u(\psi_1)$-side (again, see Figure~\ref{fi:quasi}). So there is no rule that enables us to reach the corner, and we are stuck. Our solution is to `pretend' that $u(\phi_1)$ and $u(\psi_2)$ are stable directions (we call them `quasi-stable directions'). More precisely, we ensure that the set $\T\subset S^1$ with respect to which our droplets are defined includes the directions $u(\phi_1)$ and $u(\psi_2)$. If we do this, and `grow' the $u(\phi_1)$- and $u(\psi_2)$-sides of our $\T$-droplets just as the sides with respect to stable directions, then we find that we are indeed able to grow to the corners.

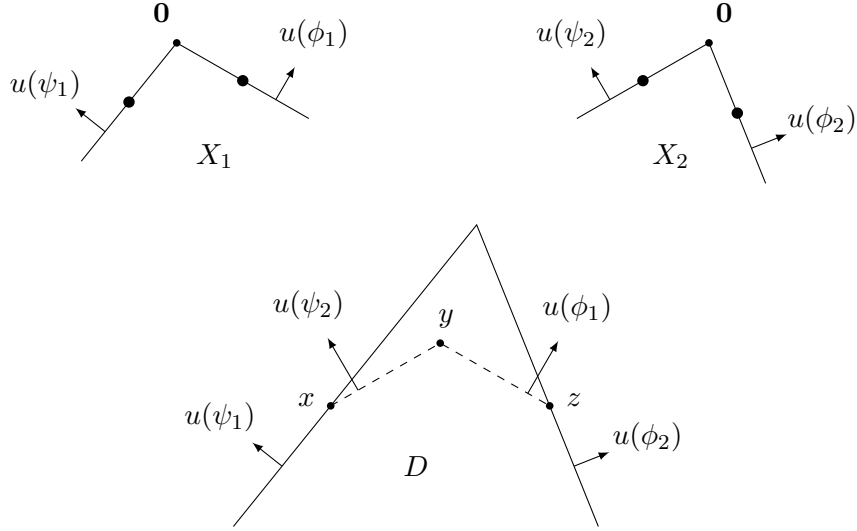
\begin{figure}[ht]
  \centering
  \begin{tikzpicture}[>=latex]
    \node [circle,fill,inner sep=0pt,minimum size=0.1cm] at (0,0) {};
    \node [circle,fill,inner sep=0pt,minimum size=0.1cm] at (7,0) {};
    \node [circle,fill,inner sep=0pt,minimum size=0.15cm] at ($(0,0)!1cm!(-1,-1.25)$) {};
    \node [circle,fill,inner sep=0pt,minimum size=0.15cm] at ($(0,0)!1cm!(-30:1)$) {};
    \path (7,0) ++($(0,0)!1cm!(-150:1)$) node [circle,fill,inner sep=0pt,minimum size=0.15cm] {};
    \node [circle,fill,inner sep=0pt,minimum size=0.15cm] at ($(7,0)!1cm!(7.5,-1.25)$) {};
    \draw (0,0) -- ($(0,0)!2cm!(-1,-1.25)$);
    \draw (0,0) -- ($(0,0)!2cm!(-30:1)$);
    \draw (7,0) -- ++($(0,0)!2cm!(-150:1)$);
    \draw (7,0) -- ++($(0,0)!2cm!(0.5,-1.25)$);
    \draw [->] ($(0,0)!1.5cm!(-1,-1.25)$) -- ++($(0,0)!0.5cm!(-1.25,1)$);
    \draw [->] ($(0,0)!1.5cm!(-30:1)$) -- ++($(0,0)!0.5cm!(60:1)$);
    \draw [->] (7,0) ++($(0,0)!1.5cm!(-150:4)$) -- ++($(0,0)!0.5cm!(120:1)$);
    \draw [->] (7,0) ++($(0,0)!1.5cm!(0.5,-1.25)$) -- ++($(0,0)!0.5cm!(1.25,0.5)$);
    \path ($(0,0)!1.5cm!(-1,-1.25)$) ++($(0,0)!1cm!(-1.25,1)$) node {$u(\psi_1)$};
    \path ($(0,0)!1.5cm!(-30:1)$) ++($(0,0)!1cm!(60:1)$) node {$u(\phi_1)$};
    \path (7,0) ++($(0,0)!1.5cm!(-150:4)$) ++($(0,0)!1cm!(120:1)$) node {$u(\psi_2)$};
    \path (7,0) ++($(0,0)!1.5cm!(0.5,-1.25)$) ++($(0,0)!1cm!(1.25,0.5)$) node {$u(\phi_2)$};
    \node at (-0.2,0.4) {$\0$};
    \node at (7.2,0.4) {$\0$};
    \node at (0.5,-1.5) {$X_1$};
    \node at (6.5,-1.5) {$X_2$};
  \end{tikzpicture} \\
  \vspace{0.5cm}
  \begin{tikzpicture}[scale=0.8,>=latex]
    \path [name path=P1] (1.6,2) -- ++(30:3);
    \path [name path=P2] (5.2,2) -- ++(150:3);
    \path [name intersections={of=P1 and P2,by=P}];
    \pgfresetboundingbox
    \draw (0,0) -- (4,5) -- (6,0);
%    \path (0,5) -- (1,5);
    \draw [dashed] (1.6,2) -- (P) -- (5.2,2);
    \draw [->] (0.8,1) -- ++($(0,0)!0.625cm!(-1,0.8)$);
    \draw [->] (5.6,1) -- ++($(0,0)!0.625cm!(1,0.4)$);
    \draw [->] ($(1.6,2)!0.25!(P)$) -- ++($(0,0)!1cm!(120:1)$);
    \draw [->] ($(5.2,2)!0.2!(P)$) -- ++($(0,0)!1cm!(60:1)$);
    \path (0.8,1) ++($(0,0)!1.3cm!(-1,0.8)$) node {$u(\psi_1)$};
    \path (5.6,1) ++($(0,0)!1.3cm!(1,0.4)$) node {$u(\phi_2)$};
    \path ($(1.6,2)!0.25!(P)$) ++($(0,0)!1.625cm!(120:1)$) node {$u(\psi_2)$};
    \path ($(5.2,2)!0.2!(P)$) ++($(0,0)!1.625cm!(60:1)$) node {$u(\phi_1)$};
    \node at (3,1) {$D$};
    \node [circle,fill,inner sep=0pt,minimum size=0.1cm] at (1.6,2) {};
    \node [circle,fill,inner sep=0pt,minimum size=0.1cm] at (P) {};
    \node [circle,fill,inner sep=0pt,minimum size=0.1cm] at (5.2,2) {};
    \node at (1.2,2.1) {$x$};
    \path (P) ++(0.1,0.4) node {$y$};
    \node at (5.6,2.1) {$z$};
  \end{tikzpicture}
\caption{An illustration of the need for quasi-stable directions. The top part of the figure shows two rules, $X_1$ and $X_2$, that destabilize overlapping intervals in $S^1$. Each of the rules consists of just two sites. In the bottom part of the figure, the two long solid lines show one corner of an $\stab$-droplet $D$. Neither $X_1$ nor $X_2$ can be used to grow to the corner. The two dashed lines show how the droplet might look if $u(\phi_1)$ and $u(\psi_2)$ were added to the stable set. Now $X_1$ can be used to grow towards $x$ and $y$ from both directions, and $X_2$ can be used to grow towards $z$ (and $y$) from both directions.}
\label{fi:quasi}
\end{figure}

\begin{lemma}\label{le:quasi}
Let $\U$ be an arbitrary update family. Then there exists a finite set $\quasi\subset S^1$ such that if $u$ and $v$ are consecutive elements of $\stab\cup\quasi$ then there exists a rule $X\in\U$ such that
\begin{equation}\label{eq:quasi}
X \subset \big(\H_u\cup\ell_u\big) \cap \big(\H_v\cup\ell_v\big).
\end{equation}
\end{lemma}

\begin{proof}
As suggested above, our construction is to add to $\stab$ every unit vector $u$ perpendicular to $x$, for every site $x\in X$ and every update rule $X\in\U$. Thus, we define
\begin{equation}\label{eq:quasiset}
\quasi := \bigcup_{X\in\U} \bigcup_{x\in X} \big\{ u\in S^1 \,:\, \< x,u \> = 0 \big\}.
\end{equation}
Now let $u$ and $v$ be consecutive elements of $\stab\cup\quasi$, and suppose that \eqref{eq:quasi} does not hold. This means that for every $X\in\U$ we have
\begin{equation}\label{eq:quasifalse}
X \not\subset \big(\H_u\cup\ell_u\big) \cap \big(\H_v\cup\ell_v\big).
\end{equation}
Choose an arbitrary unit vector $w\in S^1$ such that $\theta(u)<\theta(w)<\theta(v)$. Since $w$ is unstable, there exists an update rule $X$ such that $X\subset \H_w$. Now, $X$ satisfies \eqref{eq:quasifalse}, so without loss of generality there exists $x\in X$ such that $x\in\H_u$ and $x\notin(\H_v\cup\ell_v)$. But this implies that there is a unit vector $w'$ such that $\< x,w'\>=0$ and $\theta(u)<\theta(w')<\theta(v)$, contradicting the construction of $\quasi$.
\end{proof}

\begin{figure}[ht]
  \centering
  \begin{tikzpicture}[>=latex]
    \draw (170:1) -- (-10:5);
    \draw (60:1) -- (-120:5);
    \draw [dashed] (-170:3) -- (10:3);
    \draw [->] (-170:1.5) -- ++(100:0.5);
    \path (-170:1.5) ++(100:0.8) node {$w$};
    \draw [->] (-10:4) -- ++(80:0.5);
    \path (-10:4) ++(80:0.8) node {$u$};
    \draw [->] (-120:4) -- ++(150:0.5);
    \path (-120:4) ++(150:0.8) node {$v$};
    \node at (-10:5.5) {$\ell_u$};
    \node at (-120:5.5) {$\ell_v$};
    \node at (2,-3) {$(\H_u\cup\ell_u)\cap(\H_v\cup\ell_v)$};
    \node [circle,fill,inner sep=0,minimum size=0.08cm] at (0,0) {};
    \node [circle,fill,inner sep=0,minimum size=0.08cm] at (-2,-1.5) {};
    \draw [dashed] (0,0) -- (-2,-1.5);
    \draw [->] (-2,-1.5) -- ($(-2,-1.5)!0.5cm!(-3.5,0.5)$);
    \node at (-2,-1.9) {$x$};
    \node at (-2.4,-0.8) {$w'$};
    \node at (-0.2,0.3) {$\0$};
  \end{tikzpicture}
  \caption{Lemma~\ref{le:quasi}}
  \label{fi:quasi2}
\end{figure}

Henceforth we fix a set $\quasi$ satisfying the conclusion of Lemma~\ref{le:quasi} (for example, the set defined in \eqref{eq:quasiset}), and we call this set the \emph{quasi-stable set} and its elements \emph{quasi-stable directions}. The next lemma is a little technical to state, but it should be thought of as saying that, provided we use the quasi-stable directions, for certain directions $v$, a $\T$-droplet can grow by a single row in the direction of $v$ (all the way to the corners) with the help of a $v$-block.

\begin{lemma}\label{le:dropletnewside}
Let $\T\subset S^1$ be finite and let
\[
D := \bigcap_{u\in\T} \big(\H_u+a_u\big)
\]
be a $\T$-droplet for some set $\{a_u\in\Z^2:u\in\T\}$ such that every side of $D$ is sufficiently long. Suppose $w_1$, $v$ and $w_2$ are elements of $\T$ and are consecutive in $\T\cup\stab\cup\quasi$. Let
\[
D' := \big\{ x\in\Z^2 : \< x-a_v,v \> \leq 0 \big\} \cap \bigcap_{u\in\T\setminus\{v\}} \big(\H_u+a_u\big)
\]
be the $\T$-droplet formed from $D$ by the addition of a single new row along its $v$-side. Finally, suppose $Z$ is a $u$-block such that $Z$ is contained in the line segment $D'\setminus D$ and is at least a sufficiently large constant distance from either end of the line segment. Then
\[
D' \subset [D\cup Z].
\]
\end{lemma}

\begin{proof}
By the definition of a $u$-block, the set $[D\cup Z]$ contains all of $D'$ except for possibly at most a constant number of sites at either end of the line segment $D'\setminus D$. Suppose the conclusion of the lemma is false, so $D'\not\subset[D\cup Z]$. Without loss of generality, there is a site $z\in (D'\setminus D) \setminus [D\cup Z]$ to the left of $Z$ (as one looks in the direction of $v$), and we may assume that $z$ is in fact the first such site to the left of $Z$. Since $z\in D'\setminus D$, we have
\[
\< z-a_v,v \> = 0 \qquad \text{and} \qquad \< z-a_{w_2},w_2 \> < 0.
\]
The directions $v$ and $w_2$ are consecutive in $\T\cup\stab\cup\quasi$, so by Lemma~\ref{le:quasi} there exists $X\in\U$ such that
\[
X \subset \big( \H_v \cup \ell_v \big) \cap \big( \H_{w_2} \cup \ell_{w_2} \big).
\]
Hence,
\begin{equation}\label{eq:Xcorner}
\< x,v \> \leq 0 \qquad \text{and} \qquad \< x,w_2 \> \leq 0
\end{equation}
for all $x\in X$. From these it follows that if $y\in z+X$, then
\[
\< y-a_v,v \> \leq 0 \qquad \text{and} \qquad \< y-a_{w_2},w_2 \> < 0,
\]
and so $y\in D'\setminus D$. Moreover, $y$ does not lie to the left of $z$ on the line $D'\setminus D$, since if it did then we would have
\[
\< z+x,w_2 \> > \< z,w_2 \>,
\]
where $x=y-z\in X$, which would imply that $\< x,w_2 \> > 0$, contradicting \eqref{eq:Xcorner}. But we have proved that $z+X\subset[D\cup Z]$. Therefore, $z\in[D\cup Z]$, which contradicts our assumption that $z\notin[D\cup Z]$.
\end{proof}

\subsection{Growth of quasi-droplets}\label{se:growth}

Our final task of this section is to construct a sequence of $\T$-droplets (for some $\T\subset S^1$) that we can use as a framework for growth in the proof of the upper bound of Theorem~\ref{th:main} and in the proof of Theorem~\ref{th:super}.

In this subsection we shall assume that $\U$ is either supercritical or critical. We fix an open semicircle $C$ as follows:
\begin{itemize}
\item if $\U$ is supercritical then we choose $C$ such that $\stab\cap C=\emptyset$;
\item if $\U$ is critical then instead we choose $C$ such that $\stab\cap C$ does not contain any strongly stable directions.
\end{itemize}
In either case, the existence of such a semicircle follows from the definition of supercritical or critical. Let the left and right endpoints of $C$ be $u^l$ and $u^r$ respectively, let the midpoint of $C$ be $u^+$, and define
\begin{equation}\label{eq:stabu}
\stabu := \big((\stab\cup\quasi)\cap C\big) \cup \{u^l,u^r,-u^+\} \qquad \text{and} \qquad \stabu' := \big((\stab\cup\quasi)\cap C\big),
\end{equation}
where $\quasi$ is a set of quasi-stable directions given by Lemma~\ref{le:quasi}. In this subsection, as well as in Sections~\ref{se:upper} and~\ref{se:super}, all droplets will be $\stabu$-droplets.

In our construction, the growth of $\stabu$-droplets will take place predominantly in the $u^+$ direction, so it will be convenient to rotate the coordinate axes so that $(x,y) := xu^+ + yu^l$ for $x,y\in\R$.

Let $R\big((a,b),(c,d)\big)$ be the discrete rectangle with opposite corners at $(a,b)$ and $(c,d)$, so that
\[
R\big((a,b),(c,d)\big) := \Big\{ (x,y)\in\Z^2 \,:\, x,y\in\R,\, a\leq x\leq c,\, b\leq y\leq d \Big\}.
\]

Choose $\lambda>0$ sufficiently large and choose $\mu>0$ sufficiently large with respect to $\lambda$. (In the applications, $\lambda$ will always be a constant, but in certain cases $\mu$ will be a function of $p$, tending to infinity as $p\to 0$.) Let
\begin{equation}\label{eq:shapes}
R := R\big((0,0),(\lambda,\mu)\big) \quad \text{and} \quad S := \bigcup_{x\geq 0} R\big((0,0),(x,\mu)\big).
\end{equation}

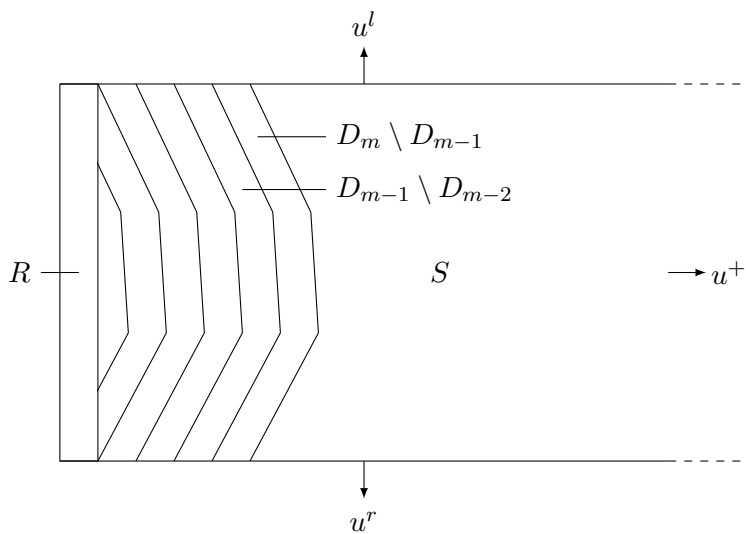
\begin{figure}[ht]
  \centering
  \begin{tikzpicture}[>=latex]
    \draw (0,0) rectangle (0.5,5);
    \draw (0,0) -- (8,0) (0,5) -- (8,5);
    \draw [dashed] (8,0) -- (9,0) (8,5) -- (9,5);
    \begin{scope}
      \clip (0.5,0) rectangle (8,5);
      \foreach \x in {0,0.5,...,2.5}
        \draw (\x,5) -- (\x+0.8,3.3) -- (\x+0.9,1.7) -- (\x,0);
    \end{scope}
    \draw (0.25,2.5) -- ++(-0.5,0) node [left] {$R$};
    \node at (5,2.5) {$S$};
    \draw [->] (4,0) -- ++(0,-0.5);
    \node at (4,-0.8) {$u^r$};
    \draw [->] (4,5) -- ++(0,0.5);
    \node at (4,5.8) {$u^l$};
%    \draw [->] (0,1.5) -- ++(-0.5,0);
%    \node at (-1,1.5) {$-u^+$};
    \draw [->] (8,2.5) -- ++(0.5,0);
    \node at (8.8,2.5) {$u^+$};
    \draw (2.6,4.3) -- (3.5,4.3) node [right] {$D_m\setminus D_{m-1}$};
    \draw (2.4,3.6) -- (3.5,3.6) node [right] {$D_{m-1}\setminus D_{m-2}$};
  \end{tikzpicture}
  \caption{The rectangle $R$, the set $S$, and the $\stabu$-droplets $(D_m)_{m\geq 0}$.}
  \label{fi:droplets}
\end{figure}

Next we define the sequence of $\stabu$-droplets. For each integer $m\geq 0$, let
\begin{equation}\label{eq:Dm}
D_m := S \cap \bigcap_{u\in\stabu'} \H_u(a_u + m d_u u),
\end{equation}
where the sets $\{a_u\in\Z^2:u\in\stabu'\}$ and $\{d_u>0:u\in\stabu'\}$ are chosen subject to the following constraints:
\begin{enumerate}
\item $D_0\subset R$;
\item for every consecutive pair of directions $u,v\in\stabu'$ there exists a line $L_u^l=L_v^r$ that is parallel to $u^+$ and intersects $S$, and is such that, for every $m\in\Z$, the intersection\footnote{These lines are discrete and therefore may have empty intersection (despite not being parallel). In this case we mean the intersection of the corresponding continuous lines.} of $\ell_u + a_u + m d_u u$ and $\ell_v + a_v + m d_v v$ lies on $L_u^l=L_v^r$;
\item for every $u\in\stabu'$ and $m\geq 0$, the $u$-side of $D_m$ has cardinality $\Omega(\mu)$.
\end{enumerate}
Note that these constraints do not uniquely specify the $D_m$, but that they are satisfied by at least one choice of the parameters $\{a_u:u\in\stabu'\}$ and $\{d_u:u\in\stabu'\}$. Any choice of parameters satisfying the constraints will be sufficient for what follows. The droplets are illustrated in Figure~\ref{fi:droplets}.

%Note that we have specifically defined the $D_m$ for negative values of $m$ as well as for non-negative values. This is because we shall need to use the $D_m$ for all values of $m$ for which the set is non-empty.

Now that we have constructed the sequence of droplets, we require one final set of definitions. We shall use these to give sufficient conditions for growth to occur from $D_{m-1}$ to $D_m$, for each $m$. For each $u\in\stabu'$ and $m\in\N$, define
\[
I(u,m) := \Big\{ i\in\Z \,:\, \ell_u(i)\cap D_{m-1} = \emptyset \text{ and } \ell_u(i)\cap D_m \neq \emptyset \Big\}.
\]
Given $u\in\stabu'$ and $i\in\Z$, a $u$-block $Z$ is said to be a \emph{$(u,i)$-block} if $Z\subset\ell_u(i)$. A $u$-block $Z$ is a \emph{suitable} $(u,i)$-block if it is a $(u,i)$-block that lies entirely between $L_u^r$ and $L_u^l$, and at least distance $\lambda$ from both, where $\lambda>0$ is sufficiently large.

The following is the main result of this subsection, and it is our main deterministic result about the growth of quasi-droplets. We shall use it in the next section to prove the upper bound of Theorem~\ref{th:main} for critical families, and in Section~\ref{se:super} to prove Theorem~\ref{th:super} for supercritical families.

\begin{lemma}\label{le:dropletgrow}
Let $m\in\N$ be such that $D_m\setminus R\neq\emptyset$, and for each $u\in\stabu'$ and each $i\in I(u,m)$, let $Z(u,i)$ be a suitable $(u,i)$-block. Then
\[
D_m \subset \bigg[ R \cup D_{m-1} \cup \bigcup_{u\in\stabu'} \bigcup_{i\in I(u,m)} Z(u,i) \bigg].
\]
\end{lemma}

\begin{figure}[ht]
  \centering
  \begin{tikzpicture}[>=latex]
    \path [name path=R2] (5.5,-5) -- (4,10);
    \path [name path=R1] (7,2) -- (3,-2);
    \path [name path=R3] (6.5,3) -- (2.5,7);
    \pgfresetboundingbox;
    \draw (0,0) -- (7,0) (0,5) -- (7,5);
    \draw [name path=S2] (6,0) -- (5.5,5);
    \draw [name path=S1] (6,0) -- (4,-2);
    \draw [name path=S3] (5.5,5) -- (3.5,7);
    \draw (5,0) -- (4.5,5);
    \draw (5,0) -- (3,-2);
    \draw (4.5,5) -- (2.5,7);
    \path [name intersections={of=R2 and S1,by=X1}];
    \path [name intersections={of=R1 and S2,by=X2}];
    \path [name intersections={of=R3 and S2,by=X3}];
    \path [name intersections={of=R2 and S3,by=X4}];
    \draw [dashed] (5,0) -- (X1);
    \draw [dashed] (5,0) -- (X2);
    \draw [dashed] (4.5,5) -- (X3);
    \draw [dashed] (4.5,5) -- (X4);
    \draw [dashed] ($(5,0)!0.2!(X2)$) -- ($(4.5,5)!0.2!(X3)$);
    \draw [dashed] ($(5,0)!0.4!(X2)$) -- ($(4.5,5)!0.4!(X3)$);
    \draw [dashed] ($(5,0)!0.6!(X2)$) -- ($(4.5,5)!0.6!(X3)$);
    \draw [->] (5.75,2.5) -- ++($(0,0)!0.5cm!(1,0.1)$);
    \path (5.75,2.5) ++($(0,0)!0.8cm!(1,0.1)$) node {$u$};
    \node at (1,0.4) {$L_u^r$};
    \node at (1,4.6) {$L_u^l$};
    \node at (4,2.5) {$D_{m-1}$};
    \draw (5.7,1.3) -- (6.5,1.3);
    \node at (7.7,1.3) {$D_m\setminus D_{m-1}$};
    \draw (4.75,4.5) -- (4,4.5) node [left] {$\rho\big(u,i(u)\big)$};
    \draw (5,3.8) -- (4,3.8) node [left] {$\rho\big(u,i(u)+1\big)$};
    \draw (5.3,4.5) -- (6,4.5) node [right] {$P$};
    \draw [->] (5,5.5) -- ++($(0,0)!0.5cm!(1,1)$);
    \path (5,5.5) ++($(0,0)!0.8cm!(1,1)$) node {$v$};
  \end{tikzpicture}
  \caption{An illustration of the growth mechanism in Lemma~\ref{le:dropletgrow}.}
  \label{fi:dropletgrow}
\end{figure}
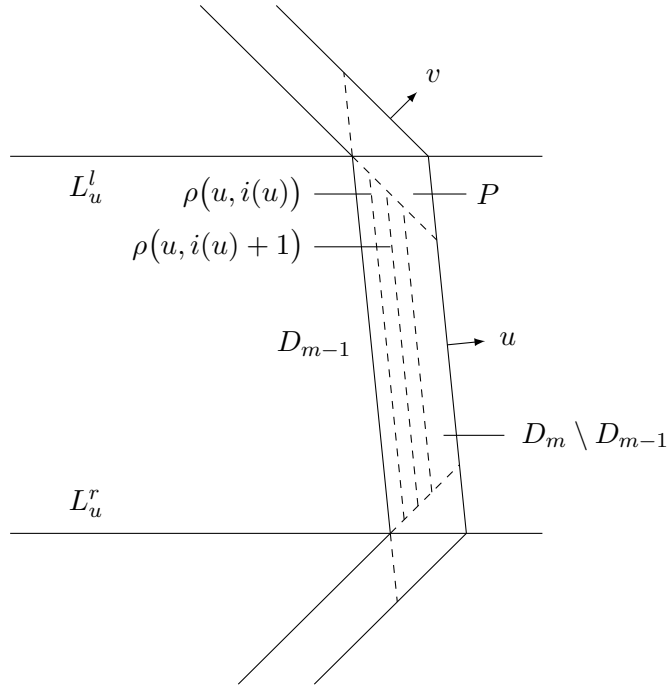

\begin{proof}
For each $u\in\stabu'$ and $v,w\in\stabu$ such that $v$, $u$, $w$ are consecutive in $\stabu$, define
\[
\rho(u,i) := \ell_u(i) \cap \H_v\big(a_v+(m-1)d_v v\big) \cap \H_w\big(a_w+(m-1)d_w w\big),
\]
where we set $d_{u^l}=d_{u^r}=0$. Also, for each $u\in\stabu'$, let $i(u):=\min I(u,m)$. Thus $\rho\big(u,i(u)\big)$ is the set of elements of the first line parallel to $\ell_u$ that lies outside of $D_{m-1}$, restricted to the line segment between the $v$- and $w$-sides of $D_{m-1}$ when those sides are extended.
%The reason for restricting to this particular line segment is that, if we are trying to grow the $u$-side of $D_{m-1}$ independently of the other sides, then this is precisely the extent to which we can apply Lemma~\ref{le:dropletnewside} to show that the growth continues to the corners of the droplet.
The situation is illustrated in Figure~\ref{fi:dropletgrow}.

We claim that the inclusion
\begin{equation}\label{eq:newclosure}
\rho\big(u,i(u)\big) \subset \Big[ R \cup D_{m-1} \cup Z\big(u,i(u)\big) \Big]
\end{equation}
follows from Lemma~\ref{le:dropletnewside}. Indeed, setting the set of directions $\T$, the $\T$-droplet $D$, the direction $v$, and the set $Z$ in the statement of Lemma~\ref{le:dropletnewside} equal to $\stabu$, $D_{m-1}$, $u$, and $Z\big(u,i(u)\big)$, respectively, it follows that the set $\rho\big(u,i(u)\big)$ we have just defined is equal to the set $D'\setminus D$ in Lemma~\ref{le:dropletnewside}, and that \eqref{eq:newclosure} holds.

By induction, after iterating the previous step $|I(u,m)|$ times, we have
\[
\bigcup_{i\in I(u,m)} \rho(u,i) \subset \bigg[ R \cup D_{m-1} \cup \bigcup_{i\in I(u,m)} Z(u,i) \bigg].
\]
Since $u\in\stabu'$ was arbitrary, it follows that
\[
\bigcup_{u\in\stabu'} \bigcup_{i\in I(u,m)} \rho(u,i) \subset \bigg[ R \cup D_{m-1} \cup \bigcup_{u\in\stabu'} \bigcup_{i\in I(u,m)} Z(u,i) \bigg].
\]
This only leaves the remaining parallelograms of sites at the corners of droplet, as shown in Figure~\ref{fi:dropletgrow}. Suppose $u$ and $v$ are consecutive directions in $\stabu'$, and let $P$ be the parallelogram at the corner between the $u$- and $v$-sides of $D_{m-1}$ and $D_m$, also as in Figure~\ref{fi:dropletgrow}. Let $X\in\U$ be a rule such that the assertion of Lemma~\ref{le:quasi} holds with directions $u$ and $v$. Choose an arbitrary unit vector $w\in S^1$ such that $\theta(u)<\theta(w)<\theta(v)$, and order the elements $x$ of $P$ in increasing order of $\< x,w\>$ (with ties resolved arbitrarily). Then $X$ can be used to infect the elements of $P$ one-by-one in this order, as required.
%Thus we only have to prove \eqref{eq:newclosure}. First, note that $Z\big(u,i(u)\big)$ is located sufficiently far from the corners of $D_{m-1}$, because it is assumed to be suitable. Thus, by the definition of a $u$-block, the infection spreads along $\ell_u\big(i(u)\big)$ in both directions from $Z\big(u,i(u)\big)$ until it is at most a constant distance from the corners, where the constant does not depend on the side lengths of the $D_m$. But then the infection continues either end of $\rho\big(u,i(u)\big)$, by Lemma~\ref{le:quasi}.
\end{proof}

\section{The upper bound for critical families}\label{se:upper}

In this section we shall use Lemma~\ref{le:dropletgrow} to prove the upper bound of Theorem~\ref{th:main} for critical update families. Throughout we assume that $\U$ is critical, that $C$ is the open semicircle specified at the start of Subsection~\ref{se:growth}, and that $\stabu$ and $\stabu'$ are as defined in \eqref{eq:stabu}. Thus, $\stab\cap C$ consists only of isolated stable directions, and therefore there exists a $u$-block for every $u\in\stabu'$. For each $u\in\stabu'$, let $\alpha_2(u)$ be the minimum cardinality of a $u$-block. Let
\begin{equation}\label{eq:alpha2}
\alpha_2 := \max\big\{ \alpha_2(u) \,:\, u\in\stabu' \big\}.
\end{equation}

We shall prove the upper bound of Theorem~\ref{th:main} in the following form. Recall once again that $\tau=\min\{t:\0\in A_t\}$.

\begin{theorem}\label{th:upper}
Let $\U$ be a critical update family, let $\epsilon>0$, and let
\[
p \geq \left(\frac{1}{\log t}\right)^{1/(\alpha_2+\epsilon)}.
\]
Then $\tau\leq t$ with high probability as $t\to\infty$.
\end{theorem}

\begin{remark}
The bound on $p$ in Theorem~\ref{th:upper} could be replaced by
\[
p \geq \frac{\lambda(\log\log t)^2}{(\log t)^{1/\alpha_2}},
\]
for a sufficiently large constant $\lambda$. However, since the bounds are likely to be far from optimal anyway, we prefer to prove the theorem with the weaker bound for reasons of simplicity.
\end{remark}

\begin{remark}\label{re:bounds}
For many natural update families, such as the family for the 2-neighbour model, there exists an arc $C\subset S^1$ of length strictly greater than $\pi$ (that is, an interval strictly larger than a semicircle) in which there are no strongly stable directions. For such families, provided the definitions of $\stabu'$ and $\alpha_2$ are changed appropriately to take account of the new set $C$, it is possible to modify our arguments to show that Theorem~\ref{th:upper} holds with the bound on $p$ replaced by
\[
p \geq \left(\frac{\lambda}{\log t}\right)^{1/\alpha_2},
\]
for a sufficiently large constant $\lambda$. In certain cases this would result in matching (up to a constant factor) bounds in Theorems~\ref{th:lower} and~\ref{th:upper} (recovering, for example, the result of~\cite{AL}). The necessary modifications to the arguments would be akin to the difference between the (asymptotically one-dimensional) growth in the upper bound of van Enter and Hulshof~\cite{vEH} and the (two-dimensional) growth in the upper bound of Aizenman and Lebowitz~\cite{AL}. We remark that in an earlier version of this paper (version 2 on the arXiv), these stronger bounds for Theorem~\ref{th:upper} were proved explicitly.
\end{remark}

We now begin the build up to the proof of Theorem~\ref{th:upper}. First we shall set out the definitions of the various quantities and shapes that we need in order to apply the framework of the previous section to the setting of critical update families. After we have done that, we shall explain how the growth mechanism will work.

Recall that $u^l$ is the left endpoint of $C$. Observe that $u^l$ is right-isolated, and therefore there exists a $u^l$-right-block, by Lemma~\ref{le:ublock}. Define $\beta$ to be the minimum cardinality of a $u^l$-right-block, and define also the following quantities:
\begin{equation}\label{eq:lengths}
\begin{split}
\mu_1(p) &= p^{-\alpha_2-\epsilon}, \\
\text{and} \qquad \mu_2(p) &= p^{-\lambda+2\beta},
\end{split}
\qquad
\begin{split}
\nu_1(p) &= p^{-\lambda}, \\
\nu_2(p) &= \exp\big(p^{-\alpha_2-2\epsilon}\big),
\end{split}
\end{equation}
where $\epsilon>0$ is arbitrary, and $\lambda>0$ is sufficiently large. Observe we have the inequalities
\[
\mu_1(p) \ll \mu_2(p) \ll \nu_1(p) \ll \nu_2(p).
\]

Recall that we have rotated the coordinate axes so that $(x,y)=xu^++yu^l$. Let $R^{(1)}$ and $R^{(2)}$ be the following rectangles:
\[
R^{(1)} := R\Big((0,0),\big(\lambda,\mu_1(p)\big)\Big) \quad \text{and} \quad R^{(2)} := R\Big(\big(\nu_1(p),0\big),\big(\nu_1(p)+\lambda,\mu_1(p)+\mu_2(p)\big)\Big).
\]
We shall use the construction for the droplets $(D_m)_{m\in\Z}$ from the previous section twice: once with $R=R^{(1)}$ and once with $R=R^{(2)}$. Thus, for each $i\in\{1,2\}$ and each integer $m\geq 0$, let $D_m^{(i)}:=D_m$, where $(D_m)_{m\geq 0}$ is the sequence of droplets defined in \eqref{eq:Dm} obtained if we set $\mu=\mu_i(p)$, and, in the case $i=2$, if we translate appropriately so that $R=R^{(2)}$ (see Figure~\ref{fi:critgrow}). Furthermore, let
\[
D^{(1)} := D_{\lambda\cdot\nu_1(p)}^{(1)} \qquad \text{and} \qquad D^{(2)} := D_{\nu_2(p)}^{(2)}.
\]
The final subset of $\Z^2$ that we need to give a name to is the triangle
\[
T := \Big\{ (x,y)\in\Z^2 \,:\, x,y\in\R, \; 0\leq x\leq \nu_1(p) \text{ and } 0\leq y-\mu_1(p) \leq \big\lceil p^{2\beta}x \big\rceil \Big\}.
\]
We shall use this triangle to grow in the direction of $u^l$.

\begin{figure}[ht]
  \centering
  \begin{tikzpicture}[>=latex]
    \draw (0,0) rectangle (0.5,3);
    \draw (6,0) rectangle (6.5,5);
    \draw (0,3) -- (6,5);
    \draw (0,0) -- (11,0) (0,3) -- (9.5,3) (6,5) -- (11,5);
    \draw [dashed] (11,0) -- (12,0) (11,5) -- (12,5);
    \foreach \x in {0.5,1,9.5}
      \draw (\x,3) -- (\x+0.8,2) -- (\x+0.9,1) -- (\x,0);
    \draw [dashed] (9.5,3) -- ++(-2.5,2);
    \node at (8.5,1.5) {$D^{(1)}$};
    \node at (4.7,1.5) {$D^{(1)}$};
    \node at (11.5,2.5) {$D^{(2)}$};
    \node at (4.7,3.7) {$T$};
    \draw (0.25,0.5) -- (-0.25,0.5) node [left] {$R^{(1)}$};
    \draw (6.25,0.5) -- (5.75,0.5) node [left] {$R^{(2)}$};
    \draw (1.4,2.5) -- (2,2.5) node [right] {$D_m^{(1)}$};
    \draw (1.32,1.8) -- (2,1.8) node [right] {$D_{m-1}^{(1)}$};
  \end{tikzpicture}
  \caption{The various subsets of $\Z^2$ involved in the growth mechanism, including the rectangles $R^{(1)}$ and $R^{(2)}$, the $\stabu$-droplets $D^{(1)}$ and $D^{(2)}$, and the triangle $T$.}
  \label{fi:critgrow}
\end{figure}
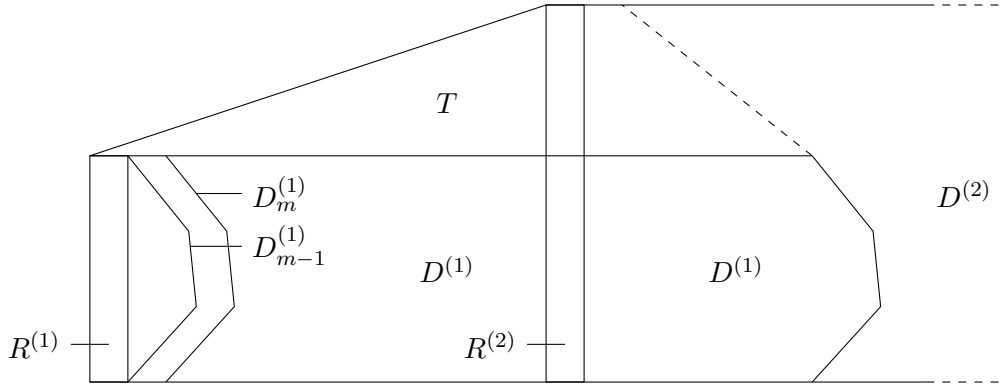

The growth mechanism that we shall use to prove Theorem~\ref{th:upper} is the following.
\begin{enumerate}
\item First, we show that there is an internally filled copy of $R^{(1)}$ such that the corresponding copy of $D^{(2)}$ contains the origin.
\item Second, given an internally filled copy of $R^{(1)}$, we use Lemma~\ref{le:dropletgrow} to show that the infection spreads in the direction of $u^+$ to fill $D^{(1)}$.
\item Next, we show that the infection spreads upwards through $T$ (that is, in the direction of $u^l$) to fill $T$ and $R^{(2)}$. Since $u^l$ may be at the end point of an interval in $\stab$, we cannot assume that there exist $u^l$-blocks, merely that there exist $u^l$-right-blocks. Thus, it could be that, as the infection spreads through $T$ row-by-row in the direction of $u^l$, one always has to look for a $u^l$-right-block \emph{to the right} of the $u^l$-right-block found in the previous row.
\item Finally, as in step (ii), we use Lemma~\ref{le:dropletgrow} to show that the infection spreads rightwards from $R^{(2)}$ to fill $D^{(2)}$, and hence infect the origin.
\end{enumerate}
In Lemma~\ref{le:Ri} we show that the events described in (ii) and (iv) each occur with high probability, and in Lemma~\ref{le:T} we show that the event described in (iii) occurs with high probability. Thus, the only unlikely event is that we find the internally filled copy of $R^{(1)}$ suitably close to the origin. It is therefore reasonable to think of $R^{(1)}$ as being a `critical droplet'.

\begin{lemma}\label{le:Ri}
For each $i\in\{1,2\}$, the event
\begin{equation}\label{eq:Ri}
\Big\{ D^{(i)} \subset \big[ R^{(i)} \cup \big(D^{(i)}\cap A\big) \big] \Big\}
\end{equation}
occurs with high probability as $p\to 0$.
\end{lemma}

\begin{proof}
Fix $i\in\{1,2\}$. Then for every $m\geq 0$ such that $D_m^{(i)}\neq\emptyset$, we have 
\[
\P_p\Big( D_m^{(i)} \subset \big[ R^{(i)} \cup D_{m-1}^{(i)} \cup \big(D_m^{(i)}\cap A\big) \big] \Big) \geq \big(1-(1-p^{\alpha_2})^{\Omega(\mu_i(p))}\big)^{O(1)},
\]
by Lemma~\ref{le:dropletgrow}. (The quantity $\Omega(\mu_i(p))$ in the exponent comes from partitioning the $u$-side of $D_m^{(i)}$ into segments of length $\alpha_2$, for each $u\in\stabu'$, and then using condition (iii) of the definition of $D_m$ immediately after \eqref{eq:Dm}.) Hence,
\begin{align}
\P_p\Big( D^{(i)} \subset \big[ R^{(i)} \cup \big(D^{(i)}\cap A\big) \big] \Big) &\geq \big(1-(1-p^{\alpha_2})^{\Omega(\mu_i(p))}\big)^{O(\nu_i(p))} \notag \\
&\geq \exp\Big( - O\big(\nu_i(p)\big) \exp\big(-\Omega(p^{\alpha_2}\mu_i(p))\big) \Big). \label{eq:Riprob}
\end{align}
Now, if $i=1$ then \eqref{eq:Riprob} is equal to
\[
\exp\Big( - O\big(p^{-\lambda}\big) \exp\big(-\Omega(p^{-\epsilon})\big) \Big) = 1-o(1),
\]
since $\mu_1(p)=p^{-\alpha_2-\epsilon}$ and $\nu_1(p)=p^{-\lambda}$. On the other hand, if $i=2$ then \eqref{eq:Riprob} is equal to
\[
\exp\Big( - O\big(\exp(p^{-\alpha_2-2\epsilon})\big) \exp\big(-\Omega(p^{-\lambda+2\beta+\alpha_2})\big) \Big) = 1-o(1),
\]
since $\mu_2(p)=p^{-\lambda+2\beta}$, $\nu_2(p)=\exp\big(p^{-\alpha_2-2\epsilon}\big)$, and $\lambda$ is sufficiently large. In either case, \eqref{eq:Ri} holds with high probability as $p\to 0$.
\end{proof}

\begin{lemma}\label{le:T}
The event
\[
\Big\{ R^{(2)} \subset \big[ D^{(1)} \cup (T\cap A) \big] \Big\}
\]
occurs with high probability as $p\to 0$.

\end{lemma}

\begin{proof}
For each $i\in\Z$, let $U_i$ be the leftmost $p^{-2\beta}$ sites of $\ell_{u^l}(i)\cap T$, and let $I$ be the set of $i\in\Z$ for which $U_i$ is non-empty. For each $i\in I$, let $Z_i\subset U_i$ be a set of $\beta$ consecutive sites contained in the middle $p^{-2\beta}/2$ sites of $U_i$.

We claim that
\begin{equation}\label{eq:R2}
R^{(2)} \subset \bigg[ D^{(1)} \cup \bigcup_{i\in I} Z(i) \bigg].
\end{equation}
To see this, first let $i_0:=\min I$, and observe that $[D^{(1)}\cup Z(i_0)]$ contains all sites of $\ell_{u^l}(i_0)$ to the right of $Z(i_0)$, at least until within $O(1)$ distance from the right-hand end of the $u^l$-side of $D^{(1)}$, since $Z(i_0)$ is a $u^l$-right-block. Losing $O(1)$ sites along each line in this way, we observe that~\eqref{eq:R2} easily holds, because the number of horizontal rows (that is, lines parallel to $\ell_{u^l}$) intersecting $T$ is at most $O(p^{-\lambda+2\beta})$, and this is much less $\Omega(p^{-\lambda})$, which is a lower bound for the number of vertical columns of sites (that is, lines parallel to $\ell_{u^+}$) intersecting $D^{(1)}$ to the right of $R^{(2)}$. (For an illustration of the situation, see Figure~\ref{fi:critgrow}, and in particular the dashed line emerging from the right-hand end of the $u^l$-side of $D^{(1)}$, which need not be parallel to any of the sides of $D^{(1)}$.)

%Now let $u\in\stabu'$ be such that $u$, $u^l$ are consecutive in $\stabu$, and hence in $\stab\cup\quasi$. Then in fact every site $x\in\ell_{u^l}(i_0)$ located to the right of $Z(i_0)$ becomes infected provided that
%\[
%x \in \H_u\big(a_u + \nu_1(p) d_u u\big),
%\]
%by Lemma~\ref{le:quasi}. (To prove this more formally, one could repeat the argument of Lemma~\ref{le:dropletnewside}.)  We therefore have by induction that, for each $i\in I$, every site in $\ell_{u^l}(i)$ located to the right of $Z(i)$ and contained in $\H_u\big(a_u + \nu_1(p) d_u u\big)$ becomes infected. Thus, since we ensured in the construction that the $u^l$-side of $D^{(1)}$ extended much further than the $u^r$-side of $T$, it follows that \eqref{eq:R2} holds.

Hence, to prove the lemma, it is sufficient to show that, with high probability, the middle $p^{-2\beta}/2$ sites of $U_i$ contain a set of $\beta$ consecutive sites, all contained in $A$, for every $i\in I$. The probability that this fails is at most
\[
p^{-\lambda+2\beta} \cdot (1-p^\beta)^{\Omega(p^{-2\beta})} \leq p^{-\lambda+2\beta} \cdot \exp\big( -\Omega(p^{-\beta}) \big) = o(1)
\]
as $p\to 0$, as required.
\end{proof}

We are now ready to complete the proof of Theorem~\ref{th:upper}.

\begin{proof}[Proof of Theorem~\ref{th:upper}]
Let
\[
p = \left(\frac{1}{\log t}\right)^{1/(\alpha_2+3\epsilon)},
\]
where $\epsilon>0$ is the same (arbitrarily small) $\epsilon$ as in the definitions of the various lengths in~\eqref{eq:lengths}. We shall show that $\tau\leq t$ with high probability as $t\to\infty$.

It will be convenient to `sprinkle' the probability in two rounds, in order to maintain independence. Formally, let $A'$ and $A''$ be two independent $p$-random subsets of $\Z^2$, and let $A:=A'\cup A''$. Thus, sites in $A$ are infected with probability $2p-p^2$, rather than $p$; this is valid (if an abuse of notation) because $\epsilon$ was chosen arbitrarily. We shall use the first round to find a suitably positioned internally filled copy of $R^{(1)}$, and the second round to show that this copy of $R^{(1)}$ grows to infect the origin by time $t$.

There are at least $\Omega\big(\nu_2(p)\big)$ sites $x\in\Z^2$ such that the sets $x+R^{(1)}$ are disjoint and $\0\in x+D^{(2)}$. The probability that for no such $x$ do we have $x+R^{(1)}\subset A'$ is at most
\[
\big(1-p^{O(\mu_1(p))}\big)^{\Omega(\nu_2(p))} \leq \exp\Big( - \Omega\big(\exp(-p^{-\alpha_2-\epsilon}) \exp(p^{-\alpha_2-2\epsilon})\big) \Big) = o(1),
\]
since $\mu_1(p)=p^{-\alpha_2-\epsilon}$ and $\nu_2(p)=\exp(p^{-\alpha_2-2\epsilon})$.

So with high probability there exists an $x$ such that $\0\in x+D^{(2)}$ and $x+R^{(1)}\subset A'$. Fix one such $x$, and for notational simplicity let us translate our notation by $-x$ so that $\0\in D^{(2)}$ and $R^{(1)}\subset A'$. By Lemmas~\ref{le:Ri} and~\ref{le:T}, the event
\[
\Big\{ D^{(2)} \subset \big[ R^{(1)} \cup \big((D^{(1)}\cup T\cup D^{(2)})\cap A''\big)\big] \Big\}
\]
occurs with high probability.

It only remains to show that $\0$ is infected by time $t$. But
\[
\big| D^{(1)}\cup T\cup D^{(2)} \big| \leq p^{-O(1)} \exp\big(p^{-\alpha_2-2\epsilon}\big) < t,
\]
so even if the sites are infected one-by-one, we still have $\0\in A_t$, which completes the proof of the theorem.
\end{proof}

\section{Supercritical families}\label{se:super}

Recall that supercritical families are those for which there exists an open semicircle in $S^1$ with which the stable set has empty intersection. Our aim in this short section of the paper is to prove Theorem \ref{th:super}. Combined with Theorem~\ref{th:main}, this provides a justification for the distinction made between supercritical and critical update families in Definition~\ref{de:class}. Using the same methods as those in the proof of Theorem~\ref{th:super}, we also establish a deterministic result, Theorem~\ref{th:superclass}, which says that an update family is supercritical if and only if there exist finite subsets of $\Z^2$ with infinite closure.

Almost all of the work that goes into the proof of Theorem~\ref{th:super} has already been completed in Section~\ref{se:upperprelims}. The only further observation required is that if $C$ is an open semicircle such that $\stab\cap C=\emptyset$ then the empty-set is a $u$-block for every $u\in C$. Thus, a sufficient condition for $\0\in A_t$ is that a sufficiently large constant sized rectangle of initially infected sites is `suitably located' within distance $t$ of the origin. (The only important detail about rectangles being `suitably located' is that there are at least $t^{\Omega(1)}$ disjoint such rectangles -- in fact, there are at least $\Omega(t)$.)

%The growth of cellular automata (in a supercritical-like fashion) has been studied by Willson \cite{Willson} and Gravner and Griffeath \cite{GG06}. Willson considers processes similar to those considered in the present paper and proves a convergence result for the asymptotic shapes of certain \emph{finite} initial configurations. It is possible to use Willson's results in place of our Proposition \ref{pr:superdet} in order to prove Theorem \ref{th:super}, although the deduction from Willson's results is far from immediate. For that reason, and because the proof of Proposition \ref{pr:superdet} is conceptually straightforward, we give a self-contained proof of Theorem \ref{th:super}, via Proposition \ref{pr:superdet}. Gravner and Griffeath study only supercritical $r$-neighbour models on undirected lattice graphs and prove stronger results in that specific case; in their context, determining for which models there exist finite sets with infinite closures is a trivial task. Thus, our advantage over \cite{GG06} is that our models are considerably more general.

\begin{proof}[Proof of Theorem \ref{th:super}]
Let $\U$ be supercritical and let $p$ satisfy $p\geq t^{-\epsilon}$, for some sufficiently small $\epsilon>0$. We shall show that $\tau\leq t$ with high probability as $t\to\infty$.

Let $\quasi$ be a set of quasi-stable directions given by Lemma~\ref{le:quasi}, let $C\subset S^1$ be the semicircle specified at the start of Subsection~\ref{se:growth}, and let $\stabu$ and $\stabu'$ be as specified in \eqref{eq:stabu}. If $u\in\stabu'$, then we have $u\in\quasi\setminus\stab$, because $\stab\cap C=\emptyset$. Therefore $u$ is a quasi-stable direction, but not a stable direction, and so the empty set is a $u$-block.

Now let $R$ and $S$ be the sets defined in \eqref{eq:shapes}, with $\mu>\lambda>0$ being sufficiently large constants, and let the $\stabu$-droplets $D_1,D_2,\dots$ be as defined subsequently in~\eqref{eq:Dm}. Then
\begin{equation}\label{eq:superinfinite}
D_m \subset [R]
\end{equation}
for every $m\in\N$, by Lemma~\ref{le:dropletgrow} and the observation we have just made that the empty set is a $u$-block for every $u\in\stabu'$.

Let $\delta>0$ be a sufficiently small constant. There are at least $\Omega(t)$ sites $x\in\Z^2$ such that $\|x\|\leq\delta t$, the rectangles $x+R$ are disjoint, and $\0\in x+S$. The probability that we do not have $x+R\subset A$ for any such $x$ is at most
\[
(1-p^{O(1)})^{\Omega(t)} \leq \exp\big( - p^{O(1)} \Omega(p^{-1/\epsilon}) \big) \to 0
\]
as $t\to\infty$, since $\epsilon$ is sufficiently small. So with high probability there exists $x\in\Z^2$ such that $\|x\|\leq\delta t$, $\0\in x+S$ and $x+R\subset A$. Suppose $x$ has this property. Then $x+D_m\subset\big[(x+D_m)\cap A\big]$ for every $m\in\N$, by \eqref{eq:superinfinite}. Now take $m$ minimal such that $\0\in x+D_m$. Then $\tau\leq|D_m|$, even if sites in $D_m$ are infected one-by-one, and also $|D_m|<t$, because $\delta$ was chosen sufficiently small. Thus $\tau\leq t$, and this completes the proof.
\end{proof}

The next result, which is another easy consequence of Lemma~\ref{le:dropletgrow} (or more specifically, of \eqref{eq:superinfinite}), is an equivalence between supercriticality and the existence of finite sets with infinite closures.

\begin{theorem}\label{th:superclass}
A two-dimensional bootstrap percolation update family $\U$ is supercritical if and only if there exists a finite set $K\subset\Z^2$ such that $[K]$ is infinite.
\end{theorem}

\begin{proof}
One implication follows from \eqref{eq:superinfinite}. The other is a triviality: if $\U$ is not supercritical, then every open semicircle intersects the stable set. Therefore there exist stable directions $u_1$, $u_2$ and $u_3$ such that the interior of their convex hull contains the origin, and hence $\{u_1,u_2,u_3\}$-droplets are finite. Thus, if $K\subset\Z^2$ is an arbitrary finite set, then $[K]\subset D$, where $D$ is the minimal $\{u_1,u_2,u_3\}$-droplet containing $K$. 
\end{proof}

Stronger statements than Theorem~\ref{th:superclass} are available for update families with zero or one stable direction(s). The proofs of these statements, which are omitted (but which may be found in version 2 of this paper on the arXiv), are further easy applications of the quasi-stability method.

\begin{theorem}
Let $\U$ be a two-dimensional bootstrap percolation update family.
\begin{enumerate}
\item There exists a finite set $K\subset\Z^2$ such that $[K]=\Z^2$ if and only if $\stab=\emptyset$.
\item Let $u\in S^1$ be a unit vector. Then there exists a finite set $K\subset\Z^2$ such that $\H_u \subset [K]$ if and only if $\stab\setminus\{u\}=\emptyset$. \qed
\end{enumerate}
\end{theorem}

\section{Open problems and conjectures}\label{se:open}

In this paper we have introduced a new, unified model of bootstrap percolation on the square lattice. Owing to the novelty of this general model, we quite naturally end the paper with a number of open problems, of which there are three broad types.

{\bf Subcritical families in two dimensions.} At present we are not able to prove any results about subcritical families. In particular, we cannot prove that the distinction we have made between critical and subcritical families is the right one, although we conjecture that it is in the following strong sense.\footnote{As noted earlier, since the submission of this paper, Balister, Bollob\'as, Przykucki and Smith \cite{BBPS} have proved this conjecture.}

\begin{conjecture}\label{co:sub}
Let $\U$ be a two-dimensional subcritical bootstrap percolation update family. Then there exists $p>0$ such that
\[
\P_p\big([A]=\Z^2\big)=0.
\]
\end{conjecture}

This conjecture lies in sharp contrast to Corollary \ref{co:pcZ2}, which says that if $\U$ is \emph{not} subcritical then for every $p>0$ we have $\P_p\big([A]=\Z^2\big)=1$. (Note that we always have $\P_p\big([A]=\Z^2\big)\in\{0,1\}$ owing to a standard $0$-$1$ law.)

{\bf Sharper results for critical families in two dimensions.} As discussed in Remark~\ref{re:bounds}, for certain update families our methods could be used with very minor modifications to obtain matching (up to a constant factor) bounds in Theorems~\ref{th:lower} and~\ref{th:upper}, and thus to determine $p_c(\Z_n^2,\U)$ up to a constant factor. However, for many other update families our methods (as they stand) do not give matching upper and lower bounds, and so for those families the order of magnitude of $p_c(\Z_n^2,\U)$ remains an open question. Moreover, it is known \cite{MountDuarte,vEH} that the order of magnitude of $p_c(\Z_n^2,\U)$ is not always a negative power of $\log n$. It is likely to be a challenging problem to determine $p_c(\Z_n^2,\U)$ up to a constant factor in general.

{\bf Higher dimensions.} It is natural to ask what can be said about the large-scale behaviour of $\U$-bootstrap percolation in dimensions $d\geq 3$. The first step towards answering this question is likely to be to give a classification of the models along the lines of Definition \ref{de:class}. This will not be a straightforward generalization of the two-dimensional definition, for the following reason. When $d=3$, the critical probability of the 2-neighbour model is $\Theta(\log n)^{-2}$ \cite{AL,BBM3D}, while the critical probability of the 3-neighbour model is $\Theta(\log\log n)^{-1}$ \cite{CerfCir,BBM3D}. By contrast, in two dimensions we have seen that the critical probabilities of all critical models are of the form $(\log n)^{-\Theta(1)}$, while the critical probabilities of non-critical models have either been proved (in the case of supercritical) or conjectured (in the case of subcritical) to be non-logarithmic. Thus, in three dimensions, the classification of critical $\U$-bootstrap models ought to divide further: at least into those for which the critical probability is $(\log n)^{-\Theta(1)}$ and those for which it is $(\log\log n)^{-\Theta(1)}$. However, this is of course only the start of the story: once that has been done, the question becomes to determine the critical probability up to a constant factor for all `critical' models in all dimensions. We anticipate that this will be a significant area of future research.

\section*{Acknowledgements}

The first author is partially supported by NSF grant DMS~1301614 and MULTIPLEX no. 317532, the second author by a CNPq bolsa PDJ, and the third author by the Knut and Alice Wallenberg Foundation. The authors are grateful to Yuval Peres and other members of the Theory Group at Microsoft Research, Redmond, where much of the research in this paper was carried out. The authors are also greatly indebted to the anonymous referee for his or her many detailed comments, which have considerably improved the presentation of this paper and allowed us to greatly streamline many of the proofs.

\bibliographystyle{amsplain}
\bibliography{../bprefs}

\end{document}